\documentclass{ert-l}

\usepackage{amsmath, amsthm, amssymb}
\usepackage[mathscr]{eucal}
\usepackage[all]{xy}

\font \boldfrak eufb10 at 12 pt

\def\bfr#1{\hbox{\boldfrak #1}}

\newtheorem{theorem}{Theorem}[section]

\newtheorem{proposition}[theorem]{Proposition}
\newtheorem{lemma}[theorem]{Lemma}
\newtheorem{corollary}[theorem]{Corollary}

\newtheorem{definition}[theorem]{Definition}

\numberwithin{equation}{section}

\DeclareMathOperator{\Hom}{Hom}
\DeclareMathOperator{\Ad}{Ad}
\DeclareMathOperator{\tr}{tr}
\DeclareMathOperator{\disc}{disc}
\DeclareMathOperator{\Hasse}{Hasse}
\DeclareMathOperator{\Int}{Int}

\begin{document}

\def\boxit#1{\vbox{\hrule
\hbox{\vrule\kern1pt
\vbox{\kern1pt#1\kern1pt}
\kern-3pt\vrule}\hrule}}

\def\J{{\mathcal J}}

\def\ladic{\overline{\Q}_\ell}
\def\cg{{\frak c}}
\def\s{{\frak s}}
\def\mg{{\frak m}}
\def\h{{\frak h}}
\def\sg{{\frak s}}
\def\tg{\mathfrak t}
\def\kbb{\mbox{k \kern-.72em k}}
\def\kbbl{{\rm k\kern-.5em k}}
\def\kbbll{{\rm k\kern-.4em k}}
\def\im{{\rm im\,}}
\def\ad{{\rm ad}}
\def\Ad{{\rm Ad}}
\def\ni{\noindent}
\def\bA{{\bf A}}
\def\bB{{\bf B}}
\def\bG{{\bf G}}
\def\bP{{\bf P}}
\def\bM{{\bf M}}
\def\bS{{\bf S}}
\def\bW{{\bf W}}
\def\bL{{\bf L}}
\def\bN{{\bf N}}
\def\bU{{\bf U}}
\def\bX{{\bf X}}
\def\cS{{\mathcal S}}
\def\bH{{\bf H}}
\def\bZ{{\bf Z}}
\def\F{{\mathbb F}}
\def\C{{\mathbb C}}
\def\bs{\backslash}
\def\cA{{\mathcal A}}
\def\cF{{\mathcal F}}
\def\bT{{\bf T}} 
\def\cB{{\mathcal B}}
\def\cJ{{\mathcal J}}
\def\cO{{\mathcal O}}
\def\ind{{\rm ind}}
\def\bGU{{\bf U}}
\def\GL{{\rm GL}}
\def\On{{\rm O}}
\def\SOn{{\rm SO}}
\def\gF{{\mathfrak F}}
\def\gO{{\mathfrak O}}
\def\gP{{\mathfrak P}}
\def\gal{{\rm Gal}}
\def\g{{\mathfrak g}}
\def\bg{{\mathbf{\frak g}}}
\def\half{\frac{1}{2}}
\def\Fq{{\mathbb F}_q}
\def\Fqq{{\mathbb F}_{q^2}}
\def\Fqb{\overline{{\mathbb F}}_q}
\def\t{\kern.1em{}^t\kern-.1em}
\def\2by2#1#2#3#4{\hbox{$\left( 
\begin{array}{cc}
{#1}&{#2}\\ {#3}&  {#4}\end{array}\right)$}}
\def\A{{\mathbb A}}
\def\Q{{\mathbb Q}}
\def\R{{\mathbb R}}
\def\cN{{\mathcal N}}
\def\Z{{\mathbb Z}}
\def\gB{\cB}
\def\ord{{\rm ord}}
\def\fW{\mathfrak{W}}
\def\fJ{\mathfrak {J}}

\def\cH{{\mathcal H}}
\def\cZ{{\mathcal Z}}
\def\cM{{\mathcal M}}
\def\cP{{\mathcal P}}
\def\cO{{\mathcal O}}
\def\vectr#1#2{\binom{#1}{#2}}

\def\Fun{{F^{un}}}
\def\b{{\frak b}}
\def\f{{\mathfrak f}}
\def\g{{\mathfrak g}}
\def\k{{\frak k}}
\def\gl{{\frak{gl}}}
\def\ind{\hbox{ind}}
\def\End{\hbox{End}}
\def\v#1{\underline{#1}}
\def\m#1{\underline{\underline{#1}}}

\title[Distinguished Supercuspidals and Odd Orthogonal Periods]{Distinguished Tame Supercuspidal Representations and Odd Orthogonal Periods}
\author{Jeffrey Hakim}
\email{jhakim@american.edu}
\address{Department of Mathematics and Statistics\\
American University\\
4400 Massachusetts Avenue, NW\\
Washington, DC 20016\\ USA}

\author{Joshua Lansky}
\email{lansky@american.edu}
\address{Department of Mathematics and Statistics\\
American University\\
4400 Massachusetts Avenue, NW\\
Washington, DC 20016\\ USA}

\thanks{Both authors were supported by NSF grant DMS-0854844.}

\begin{abstract}
We further develop and simplify the general theory of distinguished tame supercuspidal representations of reductive $p$-adic groups due to Hakim and Murnaghan, as well as the analogous theory for finite reductive groups due to Lusztig.  We apply our results to study the representations of $\GL_n(F)$, with $n$ odd and $F$ a nonarchimedean local field, that are distinguished with respect to an orthogonal group in $n$ variables.  In particular, we determine precisely when a supercuspidal representation is distinguished with respect to an orthogonal group and, if so, that the space of distinguishing linear forms has dimension one.
\end{abstract}

\maketitle

%\tableofcontents

\section{Introduction}

This paper has two objectives: (1) to further develop and simplify the general theory presented in \cite{HM} of distinguished tame supercuspidal representations, and (2) to apply this general theory to a particularly important class of examples, namely, the representations of $\GL_n(F)$, with $n$ odd and $F$ a nonarchimedean local field (under some restrictions), that are distinguished with respect to an orthogonal group in $n$ variables.

\subsection{General theory}

Generally speaking, we will use the notations and terminology of \cite{HM}.  We also impose the same ``tameness'' assumptions on the data used to define our representations.  For simplicity, we do not recall all of these conventions explicitly in this introduction.

We are interested in the tame supercuspidal representations of a given group $G= \bG (F)$, where $F$ is a nonarchimedean local field and $\bG$ is connected reductive $F$-group.  By definition, a supercuspidal representation of $G$ is ``tame''  if it is one of the representations constructed by Jiu-Kang Yu in \cite{Y}.  The basic objects used to  parametrize tame supercuspidal representations are called ``cuspidal $G$-data.'' 
(The latter notion was introduced in \cite{Y} though the terminology is from \cite{HM}.)
Assume now we have fixed a cuspidal $G$-datum $\Psi  = (\vec\bG, y,\rho,\vec\phi)$ and let $\pi (\Psi)$ denote the associated representation.

Assume we have also fixed an involution $\theta$ of $G$, that is, an $F$-automorphism of $\bG$ of order two.
The central problem considered in \cite{HM} is the computation of the dimension of $\Hom_{G^\theta} (\pi(\Psi),1)$, where $G^\theta$ is the group of fixed points of $\theta$ in $G$.  Since this dimension is constant as $\theta$ varies over its $G$-orbit $\Theta$, we write
$$\langle \Theta ,\Psi \rangle_G = \dim \Hom_{G^\theta}(\pi(\Psi),1).$$  (Recall that $G$ acts on involutions by $g\cdot \theta = \Int (g)\circ \theta \circ \Int(g)^{-1}$.)

Let $[\Psi]$ denote the set of refactorizations of $\Psi$ (in the sense of \cite{HM}) and let $[\theta]$ denote the $K^0$-orbit of $\theta$.
Note that $\vec\bG$, $y$, and the equivalence class of $\pi(\Psi)$ do not vary in the refactorization class $[\Psi]$.  On the other hand, the representation $\rho$ of $K^0$ does vary, however, its twist
$$\rho' = \rho \otimes \left( \phi |K^0\right)$$ is an invariant of $[\Psi]$.   Recall that $\phi$ is the quasicharacter of $G^0$ defined by
$$\phi = \prod_{i=0}^d \phi_i | G^0$$
and note that $\rho' |K^0_+$ is a multiple of the character $\phi | K^0_+$.

Write $[\theta] \sim [\Psi]$ when 
$\theta (K^0) = K^0$ and $\phi | K^{0,\theta}_+ =1$.
Equivalently, $[\theta] \sim [\Psi]$ when
there exists $\dot\Psi\in [\Psi]$ such that $\dot\Psi$ is $\theta$-symmetric in the sense of~\cite{HM}.
In the latter case, $\dot\Psi$ will be $\theta'$-symmetric for all $\theta'\in [\theta]$.

If $\theta (K^0 )= K^0$, we define
$$\langle [\theta],[\Psi]\rangle_{K^0} = \dim {\rm Hom}_{K^{0,\theta}}(\rho',\eta'_\theta),$$
where $\eta'_\theta$ is a certain quadratic character defined in~\cite{HM}
Otherwise, we take $\langle [\theta],[\Psi]\rangle_{K^0} = 0$.
Note that if $\langle [\theta],[\Psi]\rangle_{K^0}$ is nonzero then  $[\theta]\sim [\Psi]$.

We now state some refinements to the main theorem of \cite{HM} (Theorem 5.26) that hold under the same technical assumptions.
First of all, we prove in statement 1 of Theorem~\ref{thm:K^0_formula} that
$$\langle \Theta ,\Psi\rangle_G = \sum_{[\theta]} m_{K^0}([\theta])\ \langle [\theta],[\Psi]\rangle_{K^0},$$ where $$m_{K^0} ([\theta]) = [G_\theta : (K^0\cap G_\theta)G^\theta]$$ and we are summing over the $K^0$-orbits $[\theta]$ in $\Theta$ such that $[\theta]\sim [\Psi]$.  The fact that we are summing over $K^0$-orbits of involutions, rather than $K$-orbits, is a significant improvement over~\cite{HM} since $K^0$ has a much simpler structure than $K$.
In addition, the explicit formula defining $m_{K^0}([\theta])$ should be easy to evaluate in applications and it corrects 
a mistake in \cite{HM}.

Next, we state a formula which simplifies Theorem 5.26 (5) \cite{HM}.  Assume there exists $\theta\in \Theta$ such that $[\theta] \sim [\Psi]$ and fix such a $\theta$.
  (Such a $\theta$ always exists  if $\langle \Theta , \Psi\rangle_G$ is nonzero.)  Let $g_1,\dots, g_m\in G$ be a maximal (necessarily finite) sequence such that $g_i\theta (g_i)^{-1}\in K^0$ and the $K^0$-orbits $[g_i\cdot \theta]$ are distinct.  Then we show in statement 2 of Theorem~\ref{thm:K^0_formula} that
$$\langle \Theta ,\Psi\rangle_G = \sum_{i=1}^m m_{K^0}([g_i\cdot \theta])\ \langle [g_i\cdot \theta],[\Psi]\rangle_{K^0}.$$

We also reformulate certain results of Lusztig \cite{L} for finite groups of Lie type to make evident that our formulas for $\langle \Theta ,\Psi\rangle_G$ have close analogues for representations of finite groups of Lie type.

\subsection{A special class of examples}

Let $G = \bG (F)$, where $\bG = \GL_n$ for some odd integer $n\ge 3$ and $F$ is a nonarchimedean local field of characteristic $0$ whose residue field has characteristic $p$ with $p\ne 2$.  When $\nu\in G$ is symmetric, we define an automorphism of $\bG$ by
$$\theta_\nu (g) = \nu^{-1}\cdot {}^t g^{-1}\cdot \nu .$$
Such automorphisms will be called ``orthogonal involutions of $G$.''  In general, we use boldface letters for $F$-groups and the corresponding non-bold letters for the corresponding subgroups of $F$-rational points.  If $\theta$ is an orthogonal involution, let $\bG^\theta$ be the group of fixed points of $\theta$ (and thus $G^\theta$ denotes $\bG^\theta (F)$).

In this paper, we  compute the spaces $\Hom_{G^{\theta}}(\pi ,1)$ when $\pi$ is an irreducible tame supercuspidal representation of $G$ and $\theta$ is an orthogonal involution of $G$.  We follow the approach of \cite{HM}.  When $\Hom_{G^{\theta}}(\pi ,1)$ is nonzero, one says that $\pi$ is {\it $G^\theta$-distinguished}.  (The property of being $G^\theta$-distinguished is referred to as {\it $G^\theta$-distinction}.)

Our main theorem, Theorem 
\ref{maintheorem}, states that if
$\pi$ has central character $\omega$ then $\pi$  is $G^\theta$-distinguished precisely when $\omega(-1)=1$ and $\theta$ has the form
$\theta_\nu$ for some symmetric 
 matrix $\nu\in G$ that is similar to the matrix
 $$J  =\begin{pmatrix}
& &1\\
&\raisebox{-.1ex}{.} \cdot \raisebox{1.2ex}{.}&\\
1&&
\end{pmatrix}.$$  (Note that $\theta_{\nu_1}$ and $\theta_{\nu_2}$
are in the same $G$-orbit if
and only if $\nu_1$ is similar to a scalar multiple of $\nu_2$.)
We also show that when $\pi$ is $G^\theta$-distinguished, the dimension of $\Hom_{G^\theta} (\pi ,1)$ is one. 

Theorem \ref{maintheorem} complements work of Cesar Valverde \cite{V} whose results characterize distinction for a class of non-supercuspidal representations in the same setting we consider.  On the other hand, for depth-zero tame supercuspidal representations, the content of Theorem \ref{maintheorem} constitutes the main result in \cite{HzM}.

Our work may be used to characterize the set of tame supercuspidal representations in the image of the local metaplectic correspondence of \cite{FK} on the double cover $\widetilde{G}$ of $G$ in terms of distinguished representations.   In particular, if $\pi$ is an irreducible tame supercuspidal representation of $G$, then $\pi$ lies in the image of the Flicker-Kazhdan local metaplectic correspondence (from $\widetilde{G}$ to $G$)  if and only if $\pi$ is distinguished by any (hence every) split orthogonal group $G^\theta$.

The local results we obtain are consistent with a global conjecture of
Jacquet \cite{Ja} which suggests that, globally, a cuspidal automorphic
representation of $GL_n$ should be in the image of the metaplectic
correspondence precisely when it is distinguished, in a certain sense,
with respect to a split orthogonal similitude group.  For a precise
global statement, the reader should refer to \cite{Ja} or \cite{M}.

Let us describe the local analogue of the latter global notion of distinction.  Let $G^\theta$ be a split orthogonal group and let $G_\theta$ be the associated similitude group.  When $n$ is odd, the similitude map $\mu_\theta (g)= g\theta (g)^{-1}$ from $G_\theta$ to the center $Z$ of $G$ is surjective and we have $G_\theta = G^\theta Z$.  Let $\chi$ be a quasicharacter of $Z$.  If $\pi$ is an irreducible tame supercuspidal representation of $G$ then we say $\pi$ is {\it $(G_\theta,\chi)$-distinguished} if $\Hom_{G^\theta}(\pi,\chi\circ \mu_\theta)$ is nonzero.
Clearly, if $\pi$ is $(G_\theta ,\chi)$-distinguished then it is $G^\theta$-distinguished.  Conversely, if $\pi$ is $G^\theta$-distinguished then $\pi$ is $(G_\theta,\chi)$-distinguished precisely when the central character of $\pi$ is $\chi^2$.  (This follows immediately from the fact that $G_\theta = G^\theta Z$.)

Jacquet's conjecture is tied to a potential formulation in terms of relative trace formula of Waldspurger's work \cite{W1} \cite{W2} on  the nonvanishing at the center of symmetry of the $L$-functions attached to a quadratic twist of a  cuspidal automorphic representations of $GL_2$.
Work on Jacquet's conjecture is ongoing with contributions by various authors.  (See \cite{oO}, for example.)

We would like to also mention recent work \cite{M2} by Fiona Murnaghan that links the existence of distinguished tame supercuspidal representations for a given pair $(G,\theta)$ to the existence of certain elements in $G$ that are elliptic regular with respect to $\theta$ in a suitable sense.  The specific examples we consider are also mentioned in \cite{M2}.\\

\emph{Acknowledgements}
The authors would like to express their gratitude to Jeffrey Adler whose useful advice helped them resolve some key technical problems.

\section{General Notation and Background}
\label{sec:background}
Let $K$ be any nonarchimedean local field of characteristic $0$.  Denote by $\gO_K$ the ring
of integers of $K$ and by $\gP_K$ the maximal ideal of $\gO_K$.  Let $\f_K$ denote the residue field of $K$.  If $L/K$ is a finite extension, we let $N_{L/K}$ denote the norm map from $L^\times$ to $K^\times$.  Let $\bG$ be any connected reductive algebraic group defined over $K$.

For any subgroup $\bH$ of $\bG$, let $N_\bG(\bH)$ (resp.~$Z_\bG(\bH)$) denote the normalizer (resp.~centralizer) of $\bH$ in $\bG$.
Similarly, if $C$ is any subgroup of $\bG(K)$, we denote by $N_C(\bH)$ (resp.~$Z_K(\bH)$) the normalizer (resp.~centralizer) of $\bH$ in $C$.

Let $L$ be a finite extension of $K$.  Fix an algebraic closure $\overline K$ of $K$ containing $L$.
let $\Sigma$ denote the set of
$K$-embeddings of $L$ into $\bar K$,
and let $\iota: L\rightarrow \bar K$ denote the natural inclusion.

Let $\bH$ be an algebraic $L$-group.  Let $R_{L/K}\bH$ denote the 
$K$-group obtained from $\bH$ via the restriction of scalars functor from $L$ to $K$.  
Then there is a $\overline K$-isomorphism
$$
R_{L/K}\bH  \prod_{\sigma\in\Sigma}\bH_\sigma ,
$$
where $\bH_\sigma = \sigma(\bH)$.
The action of $\gal (\overline K/K)$ on $R_{L/K}\bH$ corresponds to the
following action on $\prod_{\sigma} \bH_\sigma$.  If $x\in
\prod_{\sigma}\bH_\sigma$, denote the $\sigma$-component of $x$ by
$x_\sigma$ for $\sigma\in\Sigma$.  Then, for $\tau\in\gal(\overline K/K)$, define $\tau\cdot x$ to be the
element of $\prod_{\sigma}\bH_\sigma$ with $\sigma$-component
$$
(\tau\cdot x)_\sigma = \tau x_{\tau^{-1}\sigma}\ \ \ (\sigma\in\Sigma).
$$
(Note that there is a natural action of
$\gal (\overline K/K)$ on $\Sigma$ so the notation $\tau^{-1}\sigma$ is
meaningful.)
We will identify $R_{L/K}\bH$ and $\prod_{\sigma\in\Sigma} \bH_\sigma$ (together with
the above action of $\gal (\overline K/K)$).  We will view each $\bH_\sigma$
as a $K$-subgroup of $G$.  Note that projection onto the
$\iota$-component gives an isomorphism of $(R_{L/K}\bH)(F)$ with $\bH(K)$.

Let $\theta$ be a $K$-involution of $\bG$.  Abusing notation slightly, we will often refer to $\theta$ as an involution of $\bG(K)$.  The group $\bG^\theta$ of $\theta$-fixed elements in $\bG$ is a reductive $K$-group.  For $g\in \bG(K)$, let $g\cdot\theta$ be the $K$-involution $\Int(g)\circ\theta\circ\Int(g)$ of $\bG$, where $\Int(g)$ is the automorphism $x\mapsto gxg^{-1}$ of $\bG$.  This defines on action of $G$ on the space of $K$-involutions of $\bG$.

We will make use of much of the above notation in the setting where the fields involved are finite.

For the remainder of the paper, $F$ will denote a fixed nonarchimedean local field of characteristic $0$.  We will abbreviate $\gO_F$, $\gP_F$, and $\f_F$ respectively by $\gO$, $\gP$, and $\f$.  Let $q$ be the cardinality of $\f$.  Let $\bG$ be any reductive $F$-group.  We will denote the group $\bG(F)$ of $F$-points of $\bG$ by $G$.  In general, we will use boldface letters to denote algebraic groups and corresponding ordinary letters to denote groups of $F$-rational points (for algebraic groups defined over $F$).  We denote the Lie algebra of $\bG$ by $\bfr g$, and set $\g = \bfr{g}(F)$.

Let $\theta$ be an involution of $G$.  Let  $\bG_\theta$ be the stabilizer of $\theta$ in $\bG$.  Then $\bG_\theta$ is a reductive $F$-group containing $\bG^\theta$.
If $\bZ$ denotes the center of $\bG$, we have 
\begin{eqnarray*}
G_\theta&=&\{ g\in G:\ g\cdot \theta = \theta\}\\
&=&\{ g\in G:\ g\theta(g)^{-1}\in Z\}.
\end{eqnarray*}
Then $g\mapsto g\theta (g)^{-1}$ gives a group homomorphism
$\mu : G_\theta \to Z$.  

When $G = \GL_n(F)$ and $G^\theta$ is an orthogonal group in $n$ variables the group $G_\theta$ is the associated orthogonal similitude group and $\mu$ is the similitude ratio.  So it is natural, in general, to view $G_\theta$ as a generalized similitude group with similitude ratio $\mu$.

The homomorphism $\mu$ yields an isomorphism of abelian groups
$$G_\theta /G^\theta \cong \mu (G_\theta),$$
as well as another such isomorphism
$$G_\theta /ZG^\theta \cong \mu (G_\theta)/\mu (Z).$$

Again, we note that much of this notation will also be used when $F$ is replaced by a fixed finite field (as in~\S\ref{sub:finite}).

Let $K$ be a finite extension of $F$.  Denote by $\cB (\bG,K)$ the Bruhat-Tits building of $\bG$ over $K$.  
For any maximal $K$-split torus $\bT$ of $\bG$, let $A(\bG,\bT,K)$ denote the apartment in $\cB (\bG,K)$ associated to $\bT$.
If $y\in\cB(\bG,K)$, let $[y]$ denote the image of $y$ in the reduced building $\cB_{\rm red}(\bG,K)$.  For any $y\in\cB(\bG,K)$, let $\bG(K)_{y,0}$ denote the associated parahoric subgroup of $\bG(K)$.  For a real number $r\geq 0$, denote by $\bG(K)_{y,r}$ the filtration subgroup of $\bG(K)_{y,0}$ attached to $y$ and $r$ by Moy and Prasad (see~\cite{MP}).  (These subgroups are defined with respect to the valuation on $K$ that restricts to the valuation on $F$ mapping $F$ onto $\Z$.)  If $K=F$, we will abbreviate $\bG(K)_{y,r}$ by $G_{y,r}$.  Let $\bG(K)_{y,r^+} = \bigcup_{s>r} \bG(K)_{y,s}$, and let $\bG(K)_{y,r:r^+} = \bG(K)_{y,r}/\bG(K)_{y,r^+}$.  The quotient $\bG(K)_{y,0:0^+}$ is the group of $\f_K$-rational points of a connected reductive $\f_K$-group which we will denote by $\mathsf G^K_y$, i.e., $\mathsf G^K_y(\f_K) = \bG(K)_{y,0:0^+} $.  When $K=F$, we will omit the superscript in this notation, i.e., $\mathsf G_y(\f) = G_{y,0:0^+}$.  The lattices $\bfr g(K)_{y,r}$, $\bfr g(K)_{y,r^+}$, and $\bfr g(K)_{y,r:r^+}$ (for $r\in\R$) are defined analogously.  When $K=F$, we will abbreviate these lattices respectively by $\g_{y,r}$, $\g_{y,r^+}$, and $\g_{y,r:r^+}$.

The following definition from \cite{HM} is derived from \cite{Y}:
\begin{definition}\label{cuspidalGdatum}
A $5$-tuple $(\vec\bG, y, \vec r,\rho,\vec\phi)$ is called a \emph{cuspidal $G$-datum} if it satisfies the following conditions:

\begin{itemize}
\item[\textbf{D1.}] $\vec{\bG}$ is a tamely ramified twisted Levi sequence $\vec{\bG}
= (\bG^0, \ldots, \bG^d)$ in $\bG$  and $\bZ^0/\bZ$
is $F$-anisotropic, where $\bZ^0$ and $\bZ$ are the centers of 
$\bG^0$ and $\bG=\bG^d$, respectively.

\item[\textbf{D2.}] $y$ is a point in $A(\bG,\bT,F)$, where $\bT$ is
a tame maximal $F$-torus of $\bG^0$ and $E'$ is a Galois tamely
ramified extension of $F$ over which $\bT$ (hence $\vec\bG$)
splits.
(Here, $A(\bG,\bT,F)=A(\bG,\bT,E')\cap \cB(\bG,F)$, where
$A(\bG,\bT,E')$ denotes the apartment in $\cB(\bG,E')$ corresponding to $\bT$.)

\item[\textbf{D3.}] $\vec{r} = (r_0, \ldots, r_d)$ is a sequence
of real numbers satisfying
$0 < r_0 < r_1 < \ldots < r_{d-1} \leq r_d$, if $d > 0$, and $0 \leq r_0$
if $d = 0$.

\item[\textbf{D4.}] $\rho$ is an irreducible representation of the stabilizer
$K^0 = G^0_{[y]}$ of $[y]$ in $G^0$ such that
$\rho\,|\, {G^0_{y,0^+}}$ is $1$-isotypic and
the compactly induced representation $\pi_{-1} = \ind_{K^0}^{G^0} \rho$
is irreducible (hence supercuspidal). Here, $[y]$ denotes the
image of $y$ in the reduced building of $G$.

\item[\textbf{D5.}] $\vec\phi = (\phi_0,\dots ,\phi_d)$ is a 
sequence of quasicharacters, where $\phi_i$ is a quasicharacter of $G^i$.  
We assume that $\phi_d =1$ if $r_d=r_{d-1}$ (with $r_{-1}$ defined to be 0),
 and in all other cases if $i\in \{\, 0,\dots,d\,\}$ 
then $\phi_i$ is trivial on $G^i_{y,r_i^+}$ 
but nontrivial on $G^i_{y,r_i}$.
\end{itemize}
\end{definition}
As observed in~\cite{HM}, the vector $\vec r$ is completely determined by $\vec\psi$.  Consequently, we can and will omit $\vec r$ and refer to the resulting $4$-tuple $(\vec\bG, y, \rho,\vec\phi)$ as a {\it cuspidal $G$-datum}.

A cuspidal $G$-datum $\Psi= (\vec\bG, y, \rho,\vec\phi)$ determines an open compact-mod-center subgroup $K = K(\Psi)$ of $G$.  As described in~\S3.1 of~\cite{HM}, $K$ can be expressed as a product $K^0J^1\cdots J^d$, where $K^0$ is as defined above, and for $i = 1,\ldots d$, $J^i$ is a certain open compact-mod-center pro-$p$ subgroup of $G^i$.  The datum $\Psi$ also determines a representation $\kappa = \kappa (\Psi)$ of $K$.

If the cuspidal $G$-datum $\Psi = (\vec\bG,y,\rho,\vec\phi)$ satisfies certain genericity conditions (namely, those in Definition~3.11 in~\cite{HM}), then the compactly induced representation $\ind_K^G\kappa$ is irreducible and hence supercuspidal.  Such data are called {\it generic} in~\cite{HM}.  For the sake of brevity, in this paper, {\it all cuspidal $G$-data will be assumed to be generic.}

Suppose $\Psi$ is a cuspidal $G$-datum.  
Let $\xi$ be the $K$-equivalence class of $\Psi$, as defined in \cite{HM}.  This consists of cuspidal $G$-data that are related to $\Psi$ by some combination of $K$-conjugation, refactorization, and ``elementary transformation,'' that is, replacing $y$ and $\rho$ by $\dot y$ and $\dot\rho$, where $[\dot y] = [y]$ and $\dot\rho\cong\rho$.  

Let $\theta$ be an involution of $G$.  
Following \cite{HM}, we say that $\Psi$ is {\it  $\theta$-symmetric} if:
\begin{itemize}
\item $\theta (\vec\bG) = \vec\bG$,
\item $\vec\phi \circ\theta = \vec\phi^{-1}$,
\item $\theta ([y]) = [y]$, where $[y]$ is the point in the reduced building of $G$ corresponding to $y$.
\end{itemize}
When the first two conditions are satisfied, but not necessarily the third condition, we say that $\Psi$ is {\it weakly $\theta$-symmetric}.

If $\Theta$ is the $G$-orbit of $\theta$, define
$$\langle\Theta,\Psi\rangle = \dim\Hom_{G^\theta}(\pi(\Psi),1).$$
Of course, this dimension is independent of the particular choice of representative $\theta$ of $\Theta$.
Also, it depends only on the $K$-equivalence class $\xi$, and so we also denote it by $\langle\Theta,\xi\rangle$.
Let $\Theta'$ be a $K$-orbit of involutions of $G$.
Then $\langle \Theta',\xi\rangle_K$ is defined in \cite{HM} by
$$\langle \Theta',\xi\rangle_K = \dim \Hom_{K^\theta} (\kappa (\Psi),1),$$ where $\theta$ is an arbitrary element of $\Theta'$.
When $\langle \Theta' ,\xi\rangle_K$ is nonzero, we say that $\Theta'$ and $\xi$ are \textit{strongly compatible.}  By Propositions 5.7 and 5.20 in \cite{HM}, if $\Theta'$ and $\xi$ are strongly compatible, then, according to Proposition 5.20 of~\cite{HM}, they are also \textit{moderately compatible}, which is equivalent to the statement that we can choose a refactorization $\dot\Psi$ of $\Psi$ and $\theta\in \Theta'$ such that $\dot\Psi$ is $\theta$-symmetric.

\section{Distinguished representations: general theory}

\subsection{A refined multiplicity formula}
\label{sec:correction} 

In this  section, we follow the notations of \cite{HM}.  Our first goal is to correct an error in \cite{HM}  (which we thank Shaun Stevens for reporting to us).    In particular, the constants $m_K(\Theta)$ do not appear to be well-defined and should be replaced by a family of constants $m_{K}(\Theta')$, as $\Theta'$ varies over the set $\Theta^K$ of $K$-orbits in $\Theta$.  This error does not affect the theory of \cite{HM} in a substantial way but it does affect some of the statements of the main results.  In particular, the formula
$$
\langle \Theta ,\xi\rangle_G = m_K(\Theta )\sum_{\Theta'\in \Theta^K} \langle \Theta' ,\xi\rangle_K
$$
which occurs throughout \cite{HM} should be replaced by
\begin{equation}\label{fixedeqn}
\langle \Theta ,\xi\rangle_G = \sum_{\Theta'\in \Theta^K} m_{K}(\Theta')\langle \Theta' ,\xi\rangle_K.
\end{equation}

A secondary purpose is to obtain formulas for the quantities in (\ref{fixedeqn}) which involve only $K^0$ and not the much more complicated group $K$.  This should greatly simplify the computations in examples.  In particular, we show that each $K$-orbit $\Theta'\subset\Theta$ contains a unique $K^0$-orbit that contains an involution $\theta$ such that $m_K(\Theta') = [G_\theta:(K^0\cap G_\theta) G^\theta]$.  The same is then true for any element of the $K^0$-orbit $[\theta]$ of $\theta$, so we denote this index by $m_{K^0}([\theta])$.  It is also shown that $m_{K^0}([\theta])$ is a power of two in general.

In addition, if $\Theta'$ contributes nontrivially to (\ref{fixedeqn}), it is shown in~\cite{HM} that
$$\langle \Theta' ,\xi\rangle_K = \dim {\rm Hom}_{K^{0,\theta}}(\rho'(\Psi),\eta'_\theta(\Psi)).$$
This formula also holds with $\theta$ replaced by any element of $[\theta]$ and $\Psi$ replaced by any datum in the class $[\Psi]$ of refactorizations of $\Psi$, and so we denote this dimension by $\langle [\theta],[\Psi]\rangle_{K^0}$.  The upshot of these results is that in \S\ref{subsub:simplified}, we show that it is possible to re-express (\ref{fixedeqn}) in the form
$$
\langle \Theta ,\Psi\rangle_G = \sum_{[\theta]} m_{K^0}([\theta])\ \langle [\theta] ,[\Psi]\rangle_{K^0},
$$
where the summation is over a certain collection of $K^0$-orbits $[\theta]\in\Theta$ depending on $[\Psi]$.

\subsubsection{The constants $m_K(\Theta')$}

From now on, we assume that we have fixed  a $G$-orbit $\Theta$ of involutions of $G$ and an inducing subgroup $K$, as in \cite{HM}.  Then $\Theta$ is a union of $K$-orbits $\Theta'$ and the set of such $K$-orbits is denoted $\Theta^K$.  
The rule $g\mapsto g\theta (g)^{-1}$ yields a bijection between $G/G^\theta$ and the space $\cS_\theta$ of elements $g\theta (g)^{-1}$ as $g$ varies over $G$.
Recall that the action of $G$ on the set of involutions of $G$ is given by $$(g\cdot \alpha)(g') = g\alpha (g^{-1} g'g)g^{-1}.$$
We have a diagram
$$\xymatrix{
G/G^\theta\ar@{<->}[rr]\ar@{>>}[dr]&
&\cS_\theta\ar@{>>}[dl]\\
&\Theta}$$ where the maps are given by:
$$\xymatrix{
gG^\theta\ar@{<->}[rr]\ar@{|->}[dr]&
&g\theta(g)^{-1}\ar@{|->}[dl]\\
&g\cdot \theta}.$$

Let ${\mathcal F}_g$ be the fiber of $g\cdot \theta\in \Theta$ in $G/G^\theta$. Then ${\mathcal F}_g = g{\mathcal F}_1$ and so there is a canonical bijection between any fiber ${\mathcal F}_g$ and the fiber ${\mathcal F}_1$.

The group $K$ acts on $G/G^\theta$ by left translations; it acts on $\cS_\theta$ by $k\cdot x = kx\theta (k)^{-1}$; and it acts on $\Theta$ by restricting the action of $G$ on the set of involutions.

The above maps are $K$-equivariant and we obtain corresponding maps on the sets of $K$-orbits:
$$\xymatrix{
K\bs G/G^\theta \ar@{<->}[rr]\ar@{>>}[dr]&
&\cS_{\theta}^K\ar@{>>}[dl]\\
&\Theta^K}\quad \xymatrix{
KgG^\theta \ar@{<->}[rr]\ar@{|->}[dr]&
&K\cdot g\theta(g)^{-1}\ar@{|->}[dl]\\
&Kg\cdot \theta}.$$

Let ${\mathcal F}^K_g$ be the fiber in $K\bs G/G^\theta \in \Theta^K$ of $\Theta'= Kg\cdot \theta$.
Let $m_K(\Theta')$ be the cardinality of ${\mathcal F}^K_g$.  It is easy to see that $m_K (\Theta')$ is finite.  (This will follow from explicit expressions for $m_K(\Theta')$ given below.)
We have
$${\mathcal F}^K_g = g {\mathcal F}_1^{g^{-1}Kg}$$ and thus
\begin{equation}
m_K (K g\cdot \theta) = m_{g^{-1} Kg} (g^{-1}Kg\cdot \theta).
\end{equation}

Lemma 2.6 \cite{HM} asserts that the numbers $m_K(\Theta')$ remain constant as $\Theta'$ varies in $\Theta^K$, but the proof appears to be erroneous.  Building on the error, the constant $m_K(\Theta)$  is defined to be the common value of the $m_K(\Theta')$'s.

To correct this mistake, we need to use the formula
$$\langle \Theta ,\xi\rangle_G = \sum_{\Theta'\in \Theta^K} m_{K}(\Theta')\langle \Theta' ,\xi\rangle_K.$$
We remark that in most common applications at most one of the summands $\langle \Theta',\xi\rangle_K$ is nonzero.  It is also common that the constants $m_K (\Theta)$ are all 1  since $G_\theta =ZG^\theta$ for some $\theta\in \Theta$.  So the error just mentioned is not easily detected by studying examples.

\subsubsection{Elementary abelian 2-groups}

Let $\Theta'$ be the $K$-orbit of $\theta$.  To establish that $m_K(\Theta')$ is a power of two, we will show that it divides the order of the group $G_\theta/ZG^\theta$, which turns out to be an elementary finite abelian 2-group.

Let
\begin{eqnarray*}Z^1_\Theta&=&\{ z\in Z :\ \theta (z) = z^{-1}\},\\
B^1_\Theta&=&\{ z\theta (z)^{-1}  :\ z\in Z\} =\mu (Z),\\
H^1_\Theta&=&Z^1_\Theta / B^1_\Theta .
\end{eqnarray*}

Fix $\theta\in \Theta$ and let ${\mathcal G}_\theta = G_\theta /ZG^\theta$.  Let ${\mathcal K}_\theta$ denote the image of $K\cap G_\theta$ in ${\mathcal G}_\theta$.  The homomorphism $\mu$ yields an isomorphism of ${\mathcal G}_\theta$ with the subgroup $\mu (G_\theta)/B^1_\Theta$ of $H^1_\Theta$. But the group $H^1_\Theta$ is an elementary finite abelian 2-group.  (See the proof of Lemma 2.8~\cite{HM}.)  It follows that ${\mathcal G}_\theta$ is an elementary finite abelian 2-group.

Now let $\Theta'$ be the $K$-orbit of $\theta$.  We have:

\begin{lemma}\label{GmodK}  The constant
$m_K(\Theta')$ is identical to the order of the elementary finite abelian 2-group ${\mathcal G}_\theta /{\mathcal K}_\theta$ and thus it is a power of 2.
\end{lemma}

\begin{proof}
The constant $m_K(\Theta')$ represents the number of elements of $K\bs G/G^\theta$ that contain a representative $g$ such that $g\cdot \theta = \theta$.  But $g\cdot \theta = \theta$ occurs exactly when $g\theta (g)^{-1}\in Z$ or, equivalently, when $g\in G_\theta$.  Hence we are counting double cosets that have a representative in $G_\theta$.

Suppose we have $Kg_1G^\theta = Kg_2 G^\theta$, with $g_1,g_2\in G_\theta$.  Then $g_2 = kg_1h$, for some $k\in K$ and $h\in G^\theta$.  Then $k$ necessarily lies in $K\cap G_\theta$.  Thus we have $(K\cap G_\theta)g_1 G^\theta = (K\cap G_\theta)g_2 G^\theta$.  It follows that there is a bijection between the set of elements of $K\bs G/G^\theta$ with a representative in $G_\theta$ and the double coset space $(K\cap  G_\theta)\bs G_\theta/ G^\theta$.  Since $G^\theta$ is a normal subgroup of $G_\theta$ and $Z\subset K\cap G_\theta$, our claim follows.
\end{proof}

We remark that Lemma~2.8 \cite{HM} establishes that $m_K(\Theta')$ is finite by showing that  
$$m_K(\Theta' ) \le |H^1_\Theta| <\infty .$$
The proof does not establish that $m_K(\Theta')$ divides $|H^1_\Theta|$ or that $m_K(\Theta')$ is a power of two.

We close this section by emphasizing that the expression just given for $m_K(\Theta')$ involves the image of $K\cap G_\theta$ in ${\mathcal G}_\theta$.  In the next section, we show that one can replace $K$ by $K^0$.

\subsubsection{A formula for $m_K(\Theta')$}
In this section, we exploit the structure of the inducing group $K$ to obtain a more precise formula for $m_K(\Theta')$.

Having defined $G_\theta$, we can now state our desired formula for $m_K(\Theta')$.  
\begin{theorem}\label{mKTheta}Let $\Psi = (\vec\bG,y,\rho,\vec\phi)$ be a generic cuspidal $G$-datum. 
Let $\xi$ be the $K$-equivalence class of $\Psi$, and let $\Theta'$ be a $K$-orbit of $F$-involutions of $\bf G$ such that $\langle \Theta',\xi\rangle_K$ is nonzero.  Then for any $\theta\in\Theta'$ such that $\Psi$ is $\theta$-symmetric,
$$m_K(\Theta') = [G_\theta : (K^0\cap G_\theta)G^\theta].$$
\end{theorem}

We will show later that there exists a unique $K^0$-orbit $\Theta'_0\subset\Theta'$ such that $\Psi$ is symmetric with respect to some (hence every) involution in $\Theta'_0$.  This justifies the notation $m_{K^0}(\Theta'_0)$ for $m_K(\Theta')$ and leads to a reformulation

\bigskip 
Fix a generic cuspidal $G$-datum $\Psi = (\vec\bG, y,\rho,\vec\phi)$.  Assume first that  $\Psi$ is $\theta$-symmetric.
Let $K = K(\Psi)$ be the inducing group associated to $\Psi$.  Then $K$ has a decomposition
$$K = K^0 J^1\cdots J^d.$$  It is shown in Proposition 3.14 \cite{HM} that all of the factors in the latter decomposition are $\theta$-stable.  Moreover, we have
$$K^\theta = K^{0,\theta } J^{1,\theta}\cdots J^{d,\theta},$$
where $S^\theta$ denotes the set of fixed points of $\theta$ in $S$.
We need a slight generalization of the latter fact.

\begin{lemma}\label{nicefactor}   If $\Psi = (\vec\bG,y,\rho,\vec\phi)$ is a 
$\theta$-symmetric cuspidal $G$-datum then $$K\cap G_\theta = (K^0\cap G_\theta) J^{1,\theta}\cdots J^{d,\theta}.$$
\end{lemma}

\bigskip The proof of the latter result is identical to that of Proposition 3.14 \cite{HM} except that, instead of Lemma 2.9 \cite{HM}, we substitute the following result whose proof is essentially the same as the proof of Lemma 2.9 \cite{HM}.

\begin{lemma}\label{lemalpha} Suppose $\alpha$ is an automorphism of a group $C$ such that $\alpha^2=1$.  Assume $A$, $B$ and $Z$ are $\alpha$-stable subgroups of $C$ such that $C=AB$ and $Z$ is a subgroup of $A$ that is contained in the center of $C$.  Let 
\begin{eqnarray*}
A'&=&\{ a\in A\ :\ a\, \alpha(a)^{-1}\in Z\},\\
B'&=&\{ b\in B\ :\  \alpha (b) = b\},\\
C'&=&\{ c\in C\ :\ c\, \alpha(c)^{-1}\in Z\}.
\end{eqnarray*}
Then $C' = A'B'$.
\end{lemma}

\bigskip  
\noindent {\it Proof of Theorem \ref{mKTheta}.} Assume first that $\Psi$ is $\theta$-symmetric.  We have shown in Lemma \ref{GmodK} that $m_K (\Theta)$ is the index of ${\mathcal K}_\theta$ in ${\mathcal G}_\theta$.  By definition, ${\mathcal K}_\theta$ is the image of $K\cap G_\theta$ in ${\mathcal G}_\theta$.  But, according to Lemma \ref{nicefactor}, $K\cap G_\theta$ is a product of $K^0\cap G_\theta$ with various pro-$p$-groups $J^{i,\theta}$.  
Since ${\mathcal G}_\theta$ is an elementary finite abelian 2-group, the groups $J^{i,\theta}$ have trivial image in ${\mathcal G}_\theta$.  
Therefore, ${\mathcal K}_\theta$ is identical to the image of $K^0\cap G_\theta$ in ${\mathcal G}_\theta$.  
Thus our claim follows when $\Psi$ is $\theta$-symmetric.

If $\Psi$ is  not necessarily  $\theta$-symmetric but $\langle \Theta',\xi\rangle_K$ is nonzero then  Proposition 5.9 and Lemma 5.19 \cite{HM} imply that there exists $\theta'\in \Theta'$ and a $\theta'$-symmetric refactorization $\dot\Psi$ of $\Psi$.  This implies that there exists $k\in K$ such that $k\cdot \dot\Psi$ is $\theta$-symmetric and our claim follows the argument in the previous paragraph.\hfill$\square$

\subsubsection{A simplified formula for $\langle\Theta,\Psi\rangle_G$}
\label{subsub:simplified}
%Under a mild condition on the residual characteristic $p$, 
Equation (\ref{fixedeqn}) can be reformulated in terms of $K^0$-orbits of involutions rather than $K$-orbits.  Since $K^0$ can have a much simpler structure than $K$, this reformulation should be regarded as a useful simplification in applications.  
%For the main results of this paper, however, this simplification is not needed and we prefer to avoid placing an additional restriction on $p$.
The idea of reducing the theory of distinguished tame supercuspidal representations to objects involving the group $\bG^0$ in the cuspidal $G$-datum is pursued further in \cite{M1}.

Suppose $\Psi = (\vec\bG,y,\rho,\vec\phi)$ is a cuspidal $G$-datum.  Then $\Psi$ determines subgroups $K^0 = K^0(\Psi)$ and $K = K(\Psi)$ as discussed above.

\begin{lemma}
\label{lem:K^0}
Let $\alpha$ be an $F$-automorphism of $\bG$.  Then $\alpha$ stabilizes $K^0$ if and only if it stabilizes both $\bG^0$ and $[y]$.
\end{lemma}
\begin{proof}
Clearly, if $\alpha$ stabilizes $\bG^0$ and $[y]\subset \cB(\bG^0,F)$, then $\alpha$ must stabilize $G^0_{[y]} = K^0$.

Conversely, suppose that $\alpha$ stabilizes $K^0$.  Then $K^0\subset G^0\cap \alpha(G^0)$.  Since $K^0$ is an open subgroup of $G^0$, it is dense in $\bG^0\cap\alpha(\bG^0)$ with respect to the Zariski topology.  (See Lemma 3.2 of~\cite{PR}.)  Thus $\bG^0\cap\alpha(\bG^0)$ must have dimension equal to that of $\bG^0$, which forces $\bG^0 = \alpha(\bG^0)$.
\end{proof}

\begin{lemma}
\label{lem:weyl}
Let $\bT^0$ be the connected component of the identity in $\bZ^0$.  Then
$$N_\bG(\bT^0)(F)\cap G_{y,0^+} = G^0_{y,0^+} .$$
\end{lemma}
\begin{proof}
Since $\bG^0 = Z_\bG(\bT^0)$, we have
$$G^0_{y,0^+}\subset Z_\bG(\bT^0)(F)\cap G_{y,0^+} \subset N_\bG(\bT^0)(F)\cap G_{y,0^+}.$$
It remains to prove the inclusion
$N_\bG(\bT^0)(F)\cap G_{y,0^+} \subset G^0_{y,0^+}$.
Since $G^0_{y,0^+} = G^0\cap G_{y,0^+}$, it is enough to show that $N_\bG(\bT^0)(F)\cap G_{y,0^+} \subset G^0$.  Moreover, 
it suffices to do this over a splitting field $E'$ of $\bT$, i.e., to show that $N_\bG(\bT^0)(E')\cap \bG(E')_{y,0^+} \subset \bG^0(E')$.
We first show that it is furthermore enough to prove the analogue of this statement in which $\bT^0$ replaced by the maximal torus $\bT$.

Let $\bT' = \Int(g)(\bT)$.  Then $\bT'$ is an $E'$-split maximal torus of $\bG^0$.  Since $g$ fixes $y$, we have $y\in A(\bG^0,\bT,E')\cap A(\bG^0,\bT',E')$.  Let $\mathsf T$ and $\mathsf T'$ be the maximal $\f_{E'}$-tori of $\mathsf G_y$ associated respectively to $\bT$ and $\bT'$ (see the appendix).  
Since the image of $g$ in $\mathsf G_y(\f_{E'})$ is trivial, we have $\mathsf T = \mathsf T'$.  It follows that there exists $h\in\bG^0(E')_{y,0^+}$ such that $\Int(h)(\bT') = \bT$.  Note that $hg\in N_\bG(\bT)(E')\cap\bG(E')_{y,0^+}$.  Thus, if we can show that 
\begin{equation}
\label{eq:inclusion}
N_\bG(\bT)(E')\cap \bG(E')_{y,0^+} \subset \bG^0(E'),
\end{equation}
it will follow that $hg\in \bG^0(E')$ so $g\in \bG^0(E')$.

It remains to prove (\ref{eq:inclusion}).
Suppose $k\in N_\bG(\bT)(E')\cap\bG(E')_{y,0^+}$.  Then the image of $k$ in $\mathsf G_y(\f_{E'})$ is trival, hence is contained in every Levi subgroup of $\mathsf G_y$ containing $\mathsf T$.  It follows that $k$ must fix pointwise every facet of $A(\bG,\bT,E')$ containing $y$.  Thus $k$ acts trivially on $A(\bG,\bT,E')$ so the image of $k$ in the Weyl group $W(\bG,\bT)$ of $\bT$ in $\bG$ is trivial.  It follows that $k\in\bT(E')\subset\bG^0(E')$, demonstrating (\ref{eq:inclusion}).
\end{proof}

\begin{proposition}
\label{prop:K^0}
Let $\theta$ be an involution of $G$ and suppose $\theta (K^0) = K^0$.  Let $k\in K$.  The following statements are equivalent.  Then $(k\cdot\theta)(K^0) = K^0$ if and only if $k\in K^0$.
\end{proposition}

\begin{proof}
It is clear that if $k\in K^0$ then $k\cdot\theta$ must stabilize $K^0$.
Conversely, suppose $\theta' = k\cdot\theta$ stabilizes $K^0$, where $k\in K$.   We have $k = k_0 j$, for some $k_0\in K^0$ and $j\in J^1\cdots J^d\subset G_{y,0^+}$.
The condition $\theta' (K^0) = K^0$ is equivalent to the condition that $k\theta (k)^{-1}$ lies in the normalizer $N_K(K^0)$ of $K^0$ in $K$.

We claim that $N_K (K^0) =K^0$.  Assume, for the moment, that this is the case.  Then 
$$k_0j\theta (j)^{-1} \theta (k_0)^{-1}\in N_K (K^0) = K^0,$$
and hence $j\theta (j)^{-1}\in K^0$.  According to Proposition 2.12 \cite{HM}, we may choose $g\in K^0_+ =G^0_{y,0^+}$ such that $j\theta (j)^{-1} = g\theta (g)^{-1}$.  Then $\theta' = \Int (k\theta (k)^{-1})\circ \theta = k'\cdot \theta$, where $k' = k_0g\in K^0$, and our claim follows.

It therefore suffices to show that $N_K(K^0)=K^0$.
Let $\bZ^0$ be the center of $\bG^0$ and let $\bT^0$ be the connected component of the identity in $\bZ^0$.
We first observe that $N_K(K^0) = N_K(\bG^0) = N_\bG(\bG^0)(F)\cap K$.  This follows from Lemma~\ref{lem:K^0} applied to the automorphisms $\Int(k)$ for $k\in K$.

We now have that We now show that  $N_\bG (K^0) = N_\bG (\bG^0) = N_\bG (\bZ^0) = N_\bG (\bT^0)$.  This is done as follows. Clearly, $N_\bG (\bZ^0)\subseteq N_\bG (\bT^0)$.  Now suppose $g\in N_{\bG}(\bT^0)$.  If $t\in \bT^0$ then $g^{-1}tg\in \bT^0$.  But if $h\in \bG^0 = Z_\bG (\bT^0)$, we have $hg^{-1}tgh^{-1} = g^{-1} t g$.  This implies $ghg^{-1}tgh^{-1} g^{-1}=t$.  Thus $ghg^{-1}\in Z_\bG (\bT^0) = \bG^0$ and so $g\in N_\bG (\bG^0)$.  This shows that $N_\bG (\bT^0)\subseteq N_\bG (\bG^0)$.  Assume next that $g\in N_\bG (\bG^0)$.  Then $\Int (g)$ is an automorphism of $\bG^0$ and thus it must preserve the center $\bZ^0$ of $\bG^0$ and its identity component $\bT^0$.
So we have shown $N_\bG (\bZ^0)\subseteq N_\bG(\bT^0)\subseteq N_\bG (\bG^0) \subseteq N_\bG (\bZ^0)$, which implies that the latter inclusions are all equalities.

We now have $N_K(K^0) = N_K(\bG^0)\cap K = N_\bG(\bT^0)\cap K$.
We are thus reduced to showing that $N_K(\bT^0) = K^0$.  Clearly, we have $N_K(\bT^0) \supset K^0$.
So suppose $k\in N_K (\bT^0)$.  As above, we write $k = k_0 j$, with $k_0\in K^0$ and $j\in J^1\cdots J^d \subset G_{y,0^+}$.  To say that $k$ normalizes $\bT^0$ is the same as saying that $j$ normalizes $\bT^0$.  But then, by Lemma~\ref{lem:weyl}, $j$ must lie in $G^0_{y,0^+}$.  Thus $k\in K^0$, and so $N_K(\bT^0) = K^0$.
\end{proof}

We now define a refinement of the $K$-equivalence on cuspidal $G$-data.  Let $[\Psi]$ denote the class of all cuspidal $G$-data related to $\Psi$ via a combination of refactorization and elementary transformation (as in \S5.1 of~\cite{HM}).  Observe that the action of an element of $K^0$ via conjugation on an element of $[\Psi]$ coincides with an elementary transformation.  Hence, $[\Psi]$ is preserved by the action of $K^0$.  We will refer to $[\Psi]$ as the \textit{refactorization class of $\Psi$.}  Note that as $\Psi$ ranges over its refactorization class, $\vec\bG$, $K$, $K^0$, and $[y]$ do not vary, while $\rho$ and $\vec\phi$ do vary.  Nevertheless, the equivalence class of the representation
$$\rho' = \rho\otimes(\phi|K^0)$$
is an invariant of the refactorization class.  Here $\phi$ is the quasicharacter of $G^0$ given by
$$\phi = \prod_{i=0}^d\phi_i|G^0.$$
Note that $\rho'|K_+^0$ is a multiple of $\phi|K_+^0$.

For an involution $\theta$ of $G$, let $[\theta]$ denote the $K^0$-orbit of $\theta$.  Consider the following two conditions on $\theta$ and $\Psi$:
\begin{enumerate}
\item $\theta$ stabilizes $K^0$.
\item $\phi|K_+^\theta =1$.
\end{enumerate}
Clearly, if (1) holds for $\theta$, then it must do so for every element of $[\theta]$.  Similarly, Lemma 5.5 of~\cite{HM} implies the analogous statement for (2).  It follows that both conditions depend only on the $K^0$-orbit $[\theta]$ and the refactorization class $[\Psi]$.  We write $[\theta]\sim[\Psi]$ when both of the above conditions hold.

\begin{proposition}
\label{prop:refactorization}
Let $\theta$ be an involution of $G$ such that $[\theta]\sim[\Psi]$.  Let $\xi$ be the $K$-equivalence class of $\Psi$ and $\Theta'$ the $K$-orbit of $\theta$.
\begin{enumerate}
\item If $\theta'\in K\cdot\theta$ and $[\theta']\sim[\Psi]$, then $[\theta']=[\theta]$.
\item If $\Psi'\in\xi$ and $[\theta]\sim [\Psi']$, then $[\Psi']=[\Psi]$.
\end{enumerate}
\end{proposition}
\begin{proof}
Let $k\in K$ be such that $\theta' = k\cdot\theta$.  Since both $\theta$ and $k\cdot\theta$ stabilize $K^0$, they must stabilize both $[y]$ and $\bG^0$ by Lemma~\ref{lem:K^0}.  It follows that $k\theta(k)^{-1}$ must stabilize $[y]$ as well, hence must lie in $K^0$.  Again, write $k = k_0j$, with $k_0\in K^0$ and $j\in J^1\cdots J^d\in G_{y,0^+}$.  Then $j\theta(j)^{-1}\in K^0\cap G_{y,0^+} = K^0_+$.  As in ~\S2.2 of \cite{HM}, let 
$$Z_\theta^1(K^0_+):=\{z\in K^0_+ : \theta(z) = z^{-1}\}.$$
Then $j\theta(j) ^{-1}\in Z_\theta^1(K^0_+)$.  But since $K^0_+$ is $2$-divisible, Lemma 2.11 of~\cite{HM} implies that $j\theta(j)^{-1} = c\,\theta(c)^{-1}$ for some $c\in K^0_+$.  Thus
$$\theta' = k\cdot\theta = \Int(k\,\theta(k)^{-1})\circ\theta = \Int (c\,\theta(c)^{-1})\circ\theta = c\cdot\theta.$$
This proves (i).

Now suppose $\Psi'\in\xi$ and $[\theta]\sim[\Psi']$.  Then there exist $k\in K$ and a $G$-datum $\dot\Psi$ in the refactorization class of $\Psi$ such that $\Psi' = {}^k\dot\Psi$.  It follows easily that $[k^{-1}\cdot\theta]\sim[\dot\Psi]$.  Since $[\theta]\sim[\Psi] = [\dot\Psi]$, (ii) implies that $[k^{-1}\cdot\theta] = [\theta]$.  Thus both $\theta$ and $k^{-1}\cdot\theta$ stabilize $K^0$.  By Proposition~\ref{prop:K^0}, it follows that $k$ must lie in $K^0$, i.e., that $\Psi' = {}^k\dot\Psi$ is in the refactorization class of $\Psi$, which proves (ii).
\end{proof}

\begin{proposition}
\label{prop:theta-symmetry}
Let $\theta$ be an involution of $G$.  Then $[\theta]\sim[\Psi]$ if and only if there exists $\dot\Psi\in[\Psi]$ such that $\dot\Psi$ is $\theta$-symmetric.  If this is the case, then $\dot\Psi$ is $\theta'$-symmetric for all $\theta'\in[\theta]$.
\end{proposition}
\begin{proof}
If $[\theta]\sim[\Psi]$, it follows from Proposition 5.9 of~\cite{HM} and Lemma~\ref{lem:K^0} that the $K$-orbit $K\cdot\theta$ and the $K$-equivalence class $\xi$ of $\Psi$ are moderately compatible.  Thus Proposition 5.7 of \emph{loc.\,cit.}~implies that there exists $\dot\Psi\in\xi$ such that $\Psi$ is $\theta$-symmetric.  Then $[\theta]\sim [\dot\Psi]$ so Proposition~\ref{prop:refactorization}(ii) implies that $\dot\Psi\in[\Psi]$.  In addition, is clear that $\dot\Psi$ must be $\theta'$-symmetric with respect to every $\theta'\in[\theta]$.

Conversely, suppose that there exists $\dot\Psi\in[\Psi]$ such that $\dot\Psi$ is $\theta$-symmetric.  Then $\theta(K^0) = K^0$.  Moreover, $K\cdot\theta$ and $\xi$ are moderately compatible by Proposition 5.7 of \emph{loc.\,cit.}  Thus $\phi|K_+^{0,\theta} = 1$ by Propositions 5.5 and 5.9 of \emph{loc.\,cit.}  Thus $[\theta]\sim[\Psi]$.
\end{proof}

Let $\Xi$ be the collection of all $(G,K^0)$-data, i.e., all $G$-data $\Psi'$ such that $K^0(\Psi') = K^0$.  Let $\theta$ be an involution of $G$, and let $\Theta'\subset\Theta$ be the $K$- and $G$-orbits containing $\theta$, respectively.  We define (see Theorem~\ref{mKTheta})
$$m_{K^0}([\theta]) := m_K(\Theta') = [G_\theta: (K^0\cap G_\theta)G^\theta].$$
Of course, this index depends only on $\Theta'$ and not on the particular choice of involution $\theta$.

Let $\Psi = (\vec\bG,y,\rho,\vec\phi) \in\Xi$.
Define
$$\langle [\theta] ,[\Psi]\rangle_{K^0} := 
\begin{cases}
\dim\Hom_{K^{0,\theta}}(\rho',\eta_\theta'), & \text{if $[\theta]\sim [\Psi]$,}\\
0, & \text{otherwise.}
\end{cases}$$
Here $\eta'_\theta$ is a certain character of $K^{0,\theta}$ of exponent two defined in \S5.6 of~\cite{HM} and described explicitly
 in~\S\ref{sec:etriv} of the present paper.  The latter definition gives a pairing between the set of $K^0$-orbits in $\Theta$ and the collection of 
 refactorization classes in $\Xi$.  Note that if some $\dot\Psi\in[\Psi]$ is $\theta$-symmetric, then
$\langle [\theta] ,[\Psi]\rangle_{K^0} = \langle \Theta' ,\xi\rangle_K$,
where $\xi$ is the $K$-equivalence class of $\Psi$.  (See~\S5.6 of~\cite{HM}.)

\begin{theorem}
\label{thm:K^0_formula}
Let $\Psi$, $\xi_0$, $K$, $K_+$, etc.~ be as above.
\begin{enumerate}
\item $\displaystyle\langle \Theta, \Psi\rangle_G = \sum_{[\theta]\sim[\Psi]} m_{K^0} ([\theta])\ \langle [\theta] , [\Psi]\rangle_{K^0}$.
\item Suppose there exists $\theta\in\Theta$ such that $[\theta]\sim[\Psi]$ (as must be the case if $\langle\Theta,\Psi\rangle_G\neq 0$).  Let $g_1,\dots, g_m\in G$ be a maximal sequence such that $g_j\theta(g_j)^{-1}\in K^0$ and the $[\theta_j]$ are distinct.  Then
$$\langle \Theta, \Psi\rangle_G = \sum_{j=1}^m m_{K^0} ([\theta_j])\ \langle [\theta_j] , [\Psi]\rangle_{K^0}.$$
\end{enumerate}
\end{theorem}
\begin{proof}
Let $\Theta'$ be a $K$-orbit of involutions of $G$ that is strongly compatible with the $K$-equivalence class $\xi$ of $\Psi$, i.e., which gives a nonzero contribution $\langle\Theta',\xi\rangle_K$ to the sum in (\ref{fixedeqn}).  Then $\Theta'$ and $\xi$ are moderately compatible by Proposition 5.20 of~\cite{HM}.  It follows from Proposition 5.9 of \emph{loc.\,cit.} that there exists $\theta\in\Theta'$ such that $[\theta]\sim[\Psi]$.  By part 1 of Proposition~\ref{prop:refactorization}, $[\theta]$ is the only $K^0$-orbit in $\Theta'$ with this property.
In addition, as discussed above, we have $\langle[\theta],[\Psi]\rangle_{K^0} = \langle\Theta',\xi\rangle_K$, and, by definition,
$m_{K^0}([\theta]) = m_K(\Theta')$.

To prove (1), it remains to show that every $K^0$-orbit $[\theta]\subset\Theta$ that gives a nonzero contribution to the right-hand side of the formula in 1 arises in the above way, i.e., is contained in some $K$-orbit $\Theta'$ of $G$-involutions that is strongly compatible with $\xi$.  For such a $K^0$-orbit $[\theta]$, we have $\langle[\theta],[\Psi]\rangle_{K^0}\neq 0$, so there must be a refactorization $\dot\Psi$ of $\Psi$ which is $\theta$-symmetric by Proposition~\ref{prop:theta-symmetry}.  As discussed above this implies that the $K$-orbit $\Theta'$ containing $\theta$ satisfies $\langle\Theta',\xi\rangle_K = \langle[\theta],[\Psi]\rangle_{K^0}\neq 0$.  This proves (1).

The first part of (2) follows directly from Theorem 5.20 in~\cite{HM} and Proposition~\ref{prop:theta-symmetry}.  Since the $[\theta_j]$ are distinct, the second part of (2) will follow from (1) provided that each refactorization class in $\Theta$ that contributes nontrivially to the sum in (1) contains one of the $\theta_j$.  Thus suppose that $\theta'\in\Theta$ satisfies $\langle[\theta'],[\Psi]\rangle_{K^0}\neq 0$.  Then $K\cdot\theta'$ and $\xi$ are moderately compatible by Proposition 5.9 of~\cite{HM} and Lemma~\ref{lem:K^0}.  It follows from Proposition 5.10 (2) that there exists $g\in G$ such that $g\theta(g)^{-1}\in K^0$ and $g\cdot \theta\in K\cdot \theta'$.  Then $[g\cdot\theta]\sim[\Psi]$.  Since $[\theta']\sim[\Psi]$, we must have $[g\cdot\theta] = [\theta']$ by Proposition~\ref{prop:refactorization}.  In other words, $\theta' = kg\cdot\theta$ for some $k\in K^0$.  If we let $h = kg$, then we have $h\theta(h)^{-1}\in K^0$.  By the maximality of $g_1,\ldots,g_m$, it follows that $[\theta'] = [h\cdot\theta]$ must contain some $\theta_j$.
\end{proof}

\subsection{Finite field theory}
\label{sub:finite}
In this section only:
\begin{itemize}
\item $\bG$ will be a connected reductive group defined over a finite field $\F_q$ of odd order $q$,
\item boldface letters will be used for $\F_q$-groups and the corresponding non-bold letters for the corresponding groups of $\F_q$-rational points.
\end{itemize}  
Fix a maximal torus $\bT$ of $\bG$ that is defined over $\F_q$ and a complex character $\lambda$ of $T$.  Let $R_\bT^\lambda = R^{\bG,\lambda}_{\bT}$ denote the  virtual representation of $G$ defined by Deligne-Lusztig \cite{DL} and let $R_{\bT,\lambda}= R_{\bT,\lambda}^\bG$ denote its virtual character.

Let $\theta$ be an involution of $G$, that is, an automorphism of $\bG$ of order 2 that is defined over $\F_q$.  Fix a   closed $\F_q$-subgroup $\bG_*^\theta$ of $\bG^\theta$ that contains the identity component of $\bG^\theta$.
If $g\in \bG$, we define the involution $g\cdot \theta$ in the usual way and we let $\bG^{g\cdot \theta}_* = g\bG^\theta_* g^{-1}$.

In \cite{L}, Lusztig gives a formula for the (virtual) dimension of the space of $G^\theta_*$-fixed points of $R_\bT^\lambda$.  We generalize this to a formula for the dimension of the space vectors in the space of $R_\bT^\lambda$ that transform under $G^\theta_*$ by a given (but arbitrary) character $\chi$ of $G^\theta_*$.

\subsubsection{A generalization of a formula of Lusztig}
\label{sec:genlusztig}

The results in this section were obtained independently by Fiona Murnaghan and appear in \cite{M2}.

If $\bH$ is an $\F_q$-group, as in \cite{L}, we define
$$\sigma (\bH) = (-1)^{\text{$\F_q$-rank of }\bH}.$$

Suppose $\bS$ is a maximal torus in $\bG$ that is defined over $\F_q$.  If $s\in \bS$, let $\bZ_s$ be the identity component of the centralizer of $s$ in $\bG$ and  let $\varepsilon_\bS: S\cap G^\theta_*\to \{ \pm 1\}$ by $$\varepsilon_\bS (s) = \sigma (Z_\bG ( (\bS\cap \bG^\theta_*)^\circ ))\  \sigma (
 Z_{\bZ_s} ( (\bS\cap \bG^\theta_*)^\circ )
) .$$  (We warn the reader that our notation $\bZ_s$ conflicts with the notations in \cite{L}.)

Let $\Xi_{\bT,\lambda,\chi}$ denote  the set of all $g\in G$ such that $(g\cdot \theta)(\bT) = \bT$ and
$$\lambda (t) = \chi (g^{-1}tg)^{-1} \varepsilon_{g^{-1}\bT g} (g^{-1}tg),$$ for all $t\in T\cap G^{g\cdot \theta}_* $.  The latter set is a union of double cosets in the space $T\bs G/G^\theta_*$.

\begin{theorem}\label{Lusztigwithchi}
If $\chi$ is a character of $G_*^\theta$ then
$$\frac{1}{|G^\theta_*|} \sum_{h\in G^\theta_*}  R_{\bT,\lambda} (h)\ \chi(h)
=\sigma (\bT) \sum_{g\in T\bs \Xi_{\bT,\lambda,\chi}/G^\theta_*} \sigma \left(Z_\bG\left( (g^{-1}\bT g \cap \bG^\theta_*)^\circ\right)\right).$$
\end{theorem}

\begin{proof}
Our proof is a routine generalization of the proof of Theorem 3.3 \cite{L}, but, since the latter proof is rather complicated, we detail the argument.  

The first step is to apply the Jordan-Chevalley decomposition to obtain:
$$\sum_{h\in G^\theta_*}  R_{\bT,\lambda} (h)\ \chi(h)
= \sum_{\genfrac{}{}{0pt}{}{s\in G^\theta_*}{\rm semisimple}} \sum_{\genfrac{}{}{0pt}{}{u\in Z_s\cap G^\theta_*}{\rm unipotent}} R_{\bT,\lambda} (su)\ \chi(su).$$ 
Since $u$ is contained in the commutator subgroup of $G^\theta_*$, we have
$$\sum_{h\in G^\theta_*}  R_{\bT,\lambda} (h)\ \chi(h)
= \sum_{\genfrac{}{}{0pt}{}{s\in  G^\theta_*}{\rm semisimple}} \chi(s) \sum_{\genfrac{}{}{0pt}{}{u\in Z_s\cap G^\theta_*}{\rm unipotent}} R_{\bT,\lambda} (su) .$$
Next, we use the Deligne-Lusztig character formula \cite{DL}
$$R_{\bT,\lambda} (su) =\frac{1}{|Z_s|} \sum_{\genfrac{}{}{0pt}{}{x\in G}{x^{-1}sx\in T}} \lambda (x^{-1} sx)\  R^{\bZ_s}_{x\bT x^{-1},1} (u).$$
 (Implicit in the latter formula is the fact that $R_{\bT,\lambda}$ is supported in the set of elements of $G$ with semisimple part in a conjugate of $T$.)
 
First, observe that Theorem 3.4 \cite{L} implies:
$$\sum_{\genfrac{}{}{0pt}{}{u\in Z_s\cap G^\theta_*}{\rm unipotent}} R_{x\bT x^{-1},1}^{\bZ_s} (u)= \frac{\sigma(\bT)}{|T|} \sum_{\genfrac{}{}{0pt}{}{g\in Z_s}{(x^{-1}g\cdot \theta)(\bT) = \bT}}\sigma ((Z_{\bZ_s}(g^{-1}x\bT x^{-1}g\cap \bZ_s\cap \bG^\theta_*)^\circ)).$$
Note that $x^{-1}sx\in T$ implies that $xTx^{-1}\subset Z_s$ and hence
$$\sigma ((Z_{\bZ_s}(g^{-1}x\bT x^{-1}g\cap \bZ_s\cap \bG^\theta_*)^\circ))
=\sigma ((Z_{\bZ_s}(g^{-1}x\bT x^{-1}g\cap \bG^\theta_*)^\circ)).$$

Let $$S = \frac{1}{|G^\theta_*|}\sum_{h\in G^\theta_*}  R_{\bT,\lambda} (h)\ \chi(h).$$
Putting the above pieces together yields
$$S=\frac{\sigma(\bT)}{|G^\theta_*||T|} \sum_{(s,x,g)} \frac{\chi(s)}{|Z_s|}\ \lambda (x^{-1} sx)\  \sigma ((Z_{\bZ_s}(g^{-1}x\bT x^{-1}g \cap \bG^\theta_*)^\circ)),
$$ where $(s,x,g)$ is summed over the set 
$$\{ (s,x,g)\in G^\theta_*\times G\times G\ : x^{-1}sx\in T,\ g\in Z_s,\ (x^{-1}g\cdot \theta )(\bT) = \bT\}.$$

We now change variables by sending $(s,x,g)$ to $(t,x',g)$, where $t= x^{-1}sx$ and $x'= g^{-1}x$.  The latter triples lie in $T\times G\times G$ subject to certain additional conditions that we now describe.  First of all, since $s = xtx^{-1} = gx' tx'^{-1} g^{-1}$, we have $gx' tx'^{-1}g^{-1}\in G^\theta_*$.  The condition $g\in Z_s$ reduces to   $g\in Z_{x'tx'^{-1}}$.  Thus the condition 
$gx' tx'^{-1}g^{-1}\in G^\theta_*$
reduces to
$x' tx'^{-1}\in G^\theta_*$.  We also have $(x'^{-1}\cdot \theta )(\bT) = \bT$.

Therefore,
$$S =\frac{\sigma(\bT)}{|G^\theta_*||T|} \sum_{(t,x',g)} \frac{\chi(x'tx'^{-1})}{|Z_{x'tx'^{-1}}|}\ \lambda (t)\  \sigma ((Z_{\bZ_{x'tx'^{-1}}}(x' \bT x'^{-1}\cap \bG^\theta_*)^\circ)),
$$
with $(t,x',g)$ summed over 
$$\{ (t,x',g)\in T\times G\times G\ : x'tx'^{-1}\in G^\theta_*,\ g\in Z_{x'tx'^{-1}},\ (x'^{-1}\cdot \theta )(\bT) = \bT\}.$$
This is the same as 
$$S=\frac{\sigma(\bT)}{|G^\theta_*||T|} \sum_{
\substack{(t,x')\in T\times G\\
x'tx'^{-1}\in G^\theta_*\\
(x'^{-1}\cdot\theta)(\bT) = \bT}
}  \chi(x'tx'^{-1}) \ \lambda (t)\  \sigma ((Z_{\bZ_{x'tx'^{-1}}}(x' \bT x'^{-1}\cap \bG^\theta_*)^\circ)).
$$
By the definition of $\varepsilon_{x'\bT x'^{-1}}$, we have $$\varepsilon_{x'\bT x'^{-1}} (x'tx'^{-1}) = \sigma (Z_\bG ( (x'\bT x'^{-1}\cap \bG^\theta_*)^\circ ))\  \sigma (
 Z_{\bZ_{x'tx'^{-1}}} ( (x'\bT x'^{-1}\cap \bG^\theta_*)^\circ )
)$$ and thus
\begin{eqnarray*}
S=\frac{\sigma(\bT)}{|G^\theta_*||T|} \sum_{
\substack{(t,x')\in T\times G\\ 
x'tx'^{-1}\in G^\theta_*\\
(x'^{-1}\cdot\theta)(\bT) = \bT}
}  \chi(x'tx'^{-1}) \ \lambda (t)\  \sigma (Z_\bG ( (x'\bT x'^{-1}\cap \bG^\theta_*)^\circ ))
\\  \cdot\  \varepsilon_{x'\bT x'^{-1}} (x'tx'^{-1}).
\end{eqnarray*}

We now change variables by replacing $(t,x')$ by $(\bar t,g)$, where $\bar t = x' tx'^{-1}$ and $g= x'^{-1}$.  This yields
\begin{eqnarray*}
S=\frac{\sigma(\bT)}{|G^\theta_*||T|} \sum_{
\substack{g \in G\\ 
(g\cdot\theta)(\bT) = \bT}
}  \sigma (Z_\bG ( (g^{-1}\bT g\cap \bG^\theta_*)^\circ ))
\\  \cdot \sum_{\bar t\in g^{-1}Tg\cap G^\theta_*}\chi(\bar t) \ \lambda (g\bar t g^{-1})\    \varepsilon_{g^{-1}\bT g} (\bar t).
\end{eqnarray*}
The sum over $\bar t$ vanishes unless $g\in \Xi_{\bT,\lambda,\chi}$ in which case it equals $|T\cap G^{g\cdot \theta}_*|$.
Hence,
$$S=\frac{\sigma(\bT)}{|G^\theta_*||T|} 
\sum_{
g\in \Xi_{\bT,\lambda,\chi}
}  \sigma (Z_\bG ( (g^{-1}\bT g\cap \bG^\theta_*)^\circ ))\cdot |T\cap G^{g\cdot \theta}_*|.
$$

Note that the above summand is constant on double cosets in $T\backslash G/G^\theta_*$.  Now let $T\times G^\theta$ act on $G$ by $(t,h)\cdot x = txh^{-1}$.  Then $TgG^\theta_*$ is the orbit of $g$.  The map $(t,h)\mapsto t$ gives a bijection between the isotropy group of $g$ and $T\cap G^{g\cdot \theta}_*$.  Thus
$$|TgG^\theta_*| = \frac{|G^\theta_*||T|}{|T\cap G^{g\cdot \theta}_*|}.$$
Therefore, $$S= \sigma(\bT) 
\sum_{
g\in T\bs \Xi_{\bT,\lambda,\chi}/G^\theta_*
}  \sigma (Z_\bG ( (g^{-1}\bT g\cap \bG^\theta_*)^\circ ))
$$
and our claim is proven.
\end{proof}

\subsubsection{Reformulation}

Let $\Theta$ be the $G$-orbit of some fixed involution $\theta_0$ of $G$. 
Above, we have assumed that $\chi$ is an arbitrary character of $G^{\theta_0}_*$.  In this section, we further require that $\chi$ can be extended to a character of $G_{\theta_0}$.  Under this assumption, 
if $g\in G$, then 
$$t\mapsto \chi (g^{-1}tg)$$
defines a character of  $T\cap G^{g\cdot\theta_0}_*$ that depends only on the involution $g\cdot\theta_0$ and not on $g$ itself.
We denote this character by $\chi_{g\cdot \theta_0}$.
Similarly,
$$t\mapsto \varepsilon_{g^{-1}\bT g} (g^{-1}tg)$$
defines a character $\varepsilon_{\bT,g\cdot \theta_0}$ of $T\cap G^{g\cdot \theta_0}_*$ depending only on $g\cdot\theta_0$.

Let $$\Theta_{\bT,\lambda,\chi} = \{ \theta \in \Theta\ : \theta (\bT) =\bT,\ \lambda | (T\cap G^{\theta}_*)= \chi_{\theta}\cdot \varepsilon_{\bT,\theta}\}.$$
Then $gG_{\theta_0}\mapsto g\cdot\theta_0$ gives a bijection between $\Xi_{\bT,\lambda,\chi}/G_{\theta_0}$ and $\Theta_{\bT,\lambda,\chi}$.  (Recall that $G_{\theta_0}$ is the stabilizer of $\theta_0$ in $G$.)  It also gives a bijection between $T\bs \Xi_{\bT,\lambda,\chi}/G_{\theta_0}$ and the space of $T$-orbits in $\Theta_{\bT,\lambda ,\chi}$.

If $\theta\in \Theta$ then we let $[\theta]$ denote the $T$-orbit of $\theta$ and we take
$$m_T ([\theta]) = [ G_{\theta} : G^{\theta}_* (G_{\theta} \cap T) ].$$
We write $[\theta]\sim \lambda$ when $\theta (\bT) = \bT$ and $\lambda |((\cap G^\theta_*) = \chi_\theta\cdot\varepsilon_{\bT,\theta}$ or, in other words, $[\theta]\subset \Theta_{\bT,\lambda ,\chi}$.
Define
$$
\langle [\theta],\lambda\rangle^\chi_T = \begin{cases}\sigma (\bT)\ \sigma (Z_\bG ((\bT \cap \bG^{\theta}_*)^\circ )),&\text{if }[\theta]\sim  \lambda ,\\
0,&\text{otherwise.}
\end{cases}
$$
Let 
$$\langle \Theta ,\lambda\rangle_G^\chi = 
\frac{1}{|G^{\theta_0}_*|} \sum_{h\in G^{\theta_0}_*}  R_{\bT,\lambda} (h)\ \chi(h).$$

\begin{theorem}
$$\langle \Theta ,\lambda\rangle_G^\chi = \sum_{[\theta]\sim \lambda } m_T([\theta] )\ \langle [\theta],\lambda\rangle_T^\chi .$$
\end{theorem}

\begin{proof}
Since the set $\Xi_{\bT,\lambda,\chi}$ may be expressed as
$$\{ g\in G\ : (g\cdot \theta_0)(\bT) = \bT,\ \lambda |(T\cap G^{g\cdot \theta_0}_*) = \chi_{g\cdot\theta_0}\cdot \varepsilon_{T,g\cdot\theta_0}\},$$
it follows that $\Xi_{\bT,\lambda,\chi}$ is a union of double cosets in $T\bs G/ G_{\theta_0}$.
If $g\in G^{\theta_0}$ then $gG^{\theta_0}_* g^{-1} = G^{g\cdot \theta_0}_* = G^{\theta_0}_*$ and thus $G^{\theta_0}_*$ is a normal subgroup of $G_{\theta_0}$.  Hence, we have an action of $G_{\theta_0}/G^{\theta_0}_*$ on $T\bs \Xi_{\bT,\lambda,\chi}/G^{\theta_0}_*$ by
$$h\cdot (T g G^{\theta_0}_*) = Tgh^{-1}G^{\theta_0}_*.$$
The isotropy group of $TgG^{\theta_0}_*$ is $(G_{\theta_0} \cap g^{-1}Tg)/(G^{\theta_0}_*\cap g^{-1}Tg)$.

We have a projection $$T\bs \Xi_{\bT,\lambda,\chi} /G^{\theta_0}_* \to T\bs \Xi_{\bT,\lambda, \chi} /G_{\theta_0} .$$  
The $G_{\theta_0}$-orbit of $TgG^{\theta_0}_*$ is the fiber of the double coset $TgG_{\theta_0}$.  The cardinality of this fiber is
$$[ G_{\theta_0} / G^{\theta_0}_* : (G_{\theta_0} \cap g^{-1}Tg)/(G^{\theta_0}_*\cap g^{-1}Tg) ]$$ or, equivalently,
$$[ G_{\theta_0} : G^{\theta_0}_* (G_{\theta_0} \cap g^{-1}Tg) ].$$
This is also the same as
$$[ G_{g\cdot \theta_0} : G^{g\cdot \theta_0}_* (G_{g\cdot \theta_0} \cap T) ].$$

We observe
$$
\sigma \left(Z_\bG\left( (g^{-1}\bT g \cap \bG^{\theta_0}_*)^\circ\right)\right)
=\sigma \left(Z_\bG\left( (\bT \cap \bG^{g\cdot\theta_0}_*)^\circ\right)\right),$$
and thus by Theorem~\ref{Lusztigwithchi},
\begin{eqnarray*}
\langle \Theta ,\lambda\rangle_G^\chi &=&\sigma (\bT) \sum_{g\in T\bs \Xi_{\bT,\lambda,\chi}/G^{\theta_0}_*} \sigma \left(Z_\bG\left( (\bT \cap \bG^{g\cdot\theta_0}_*)^\circ\right)\right)\\
&=&\sigma (\bT) \sum_{g\in T\bs \Xi_{\bT,\lambda,\chi}/G_{\theta_0}} 
[ G_{g\cdot \theta_0} : G^{g\cdot \theta_0}_* (G_{g\cdot \theta_0} \cap T) ]
\ \ \sigma \left(Z_\bG\left( (\bT \cap \bG^{g\cdot\theta_0}_*)^\circ\right)\right)\\
&=& \sum_{[\theta]\in \Theta^T_{\bT,\lambda,\chi}} m_T([\theta])\ \langle [\theta] ,\lambda\rangle_T^\chi .
\end{eqnarray*}
\end{proof}

\section{Parameters for tame supercuspidal representations of $\GL_n(F)$}
From now on, unless otherwise specified, we assume that $\bG$ is the group $\GL_n$.

\subsection{Howe data}

We recall some basic terminology and facts associated with 
Howe's construction \cite{rH} of tame supercuspidal
representations of $G = \GL_n (F)$,
and then we describe how the latter construction fits within Yu's framework of constructing tame supercuspidal representations for more general groups \cite{Y}.  A more detailed discussion of these matters is contained in \cite{HM}.

For the purposes of this paper, we find it convenient to introduce the notion of a ``Howe datum.''  This is a $\GL_n$ -variant of the notion of a cuspidal $G$-datum (in the sense of \cite{HM}).

\begin{definition}
If $E$ is a tamely ramified extension of $F$ of
degree $n$ and $\varphi$ is a quasicharacter of $E^\times$ then $\varphi$ is {\bf $F$-admissible}
 (or {\bf admissible over $F$}) if 
\begin{itemize}
\item there does not exist a proper subfield $L$ of
$E$ containing $F$ such that $\varphi$ factors through
the norm map $N_{E/L}:E^\times\rightarrow L^\times$;
\item if $L$ is a subfield of $E$ containing
$F$ and $\varphi\,|\, (1+\gP_E)$ factors through
$N_{E/L}$, then $E$ is unramified over $L$.
\end{itemize}
If $\varphi$ and $\varphi^\prime$ are $F$-admissible quasicharacters
of $E^\times$ and $E^{\prime\times}$, respectively,
then $\varphi$ and
$\varphi^\prime$ are  {\bf $F$-conjugate}
if there exists an $F$-isomorphism of $E$ with $E^\prime$
that takes $\varphi$ to $\varphi^\prime$.
\end{definition}

Howe's construction yields a bijection between the set of equivalence classes of tame supercuspidal representations of $G$ and the set of $F$-conjugacy classes of $F$-admissible quasicharacters associated to tamely ramified extensions of $F$ of degree $n$.

\begin{definition}
If $F'$ is a finite tamely ramified extension
of $F$
and $\varphi$ is a quasicharacter of $F^{\prime\times}$,
the {\bf conductoral exponent} $f(\varphi)$ of
$\varphi$ is the smallest positive integer such that
$\varphi\,|\, 1+\gP_{F^\prime}^{f(\varphi)}=1$.
\end{definition}

When 
 $F'$ is a finite tamely ramified extension of
$F$, we let $C_{F^{\prime}}$ denote the subgroup of
$F^{\prime\times}$ generated by   the roots of unity in $\gO_{F^\prime}^\times$
of order relatively prime to $p$ and by a prime element $\varpi_{F^\prime}$ in $F^\prime$
such  that $\varpi_{F^\prime}^e$ belongs
to $F$,
where $e$ is the ramification index of $F'$
over $F$.
If $\psi^\prime$ is
 a character of $F'$ that is trivial on $\gP_{F^\prime}$ 
and
nontrivial on $\gO_{F^\prime}$ and if $f(\varphi)>1$, then there exists a unique
$$
\gamma_\varphi\in C_{F^\prime}\cap
(\gP_{F^\prime}^{1-f(\varphi)}-\gP_{F^\prime}^{2-
f(\varphi)})
$$
such that $\varphi(1+t)=\psi^\prime(\gamma_\varphi t)$, 
$t\in \gP_{F^\prime}^{f(\varphi)-1}$.

\begin{definition}
Let $F'$ be a tamely ramified extension of $F$ and
let $\varphi$ be a quasicharacter of $F^{\prime\times}$.
If $f(\varphi)>1$, we say that
$\varphi$ is {\bf generic} over $F$ if $F[\gamma_\varphi]=F^\prime$.
If
 $f(\varphi)=1$, then we say that $\varphi$ is {\bf generic} over $F$ 
if $\varphi$ is $F$-admissible.
\end{definition}

We remark that if $f(\varphi)=1$ then $\varphi$
is generic over $F$ precisely when
$F'$ is unramified over $F$ and $\varphi$ is not fixed by
any nontrivial element of the Galois group $\gal(F' /F)$.
We also observe  that, in general, if $\varphi$ is generic over $F$ then it is necessarily admissible over $F$.

Let $E$ be a tamely ramified extension of $F$ of degree
$n$, and let $\varphi$ be an $F$-admissible quasicharacter
of $E^\times$.

\begin{definition}
 A  {\bf Howe factorization of $\varphi$}
consists of 
\begin{itemize}
\item a tower of fields
$F=E_d\subsetneq E_{d-1}\subsetneq\cdots\subsetneq E_0\subset E$,
with $d\ge 0$, 
\item
a collection of quasicharacters $\varphi_i$, $i=-1,
\dots, d$, 
\end{itemize}
with the following properties:
\begin{itemize}
\item For each $i\in \{\, 0,\dots,d\,\}$,
 $\varphi_i$
is a quasicharacter of $E_i^\times$ such that the conductoral
exponent $f_i=f(\varphi_i\circ N_{E/E_i})$ of $\varphi_i\circ N_{E/E_i}$
 is greater than $1$, and such that
$\varphi_i$ is generic over $E_{i+1}$
if $i\not=d$. 
\item  $f_0<f_1<\cdots <f_{d-1}$.
\item  If $\varphi_d$ is nontrivial,
then $f_d>f_{d-1}$.
\item {\rm ({\bf The toral case})} If $E_0=E$, then $\varphi_{-1}$ is the trivial
character of $E^\times$.
\item {\rm ({\bf The nontoral case})}  If $E_0\subsetneq E$, then 
$\varphi_{-1}$ is a quasicharacter
of $E^\times$ such that $f(\varphi_{-1})=1$
 and $\varphi_{-1}$ is generic over $E_0$.
\item
 $\varphi=\varphi_{-1}\prod_{i=0}^d\varphi_i\circ N_{E/E_i}$.
 \end{itemize}
\end{definition}

Note that $E/E_0$ is always unramified.

\begin{definition}\label{Howedatum} A  {\bf Howe datum (with respect to $G$)} consists of:
\begin{itemize}
\item a degree $n$ tamely ramified extension $E$ of $F$,
\item an $F$-admissible quasicharacter $\varphi : E^\times \to\C^\times$,
\item a Howe factorization of $\varphi$,
\item an $F$-linear embedding of $E$ in $M(n,F)$.
\end{itemize}
\end{definition}

The latter two ingredients affect the construction but not the equivalence class of the representation that is constructed.  If $\Phi$ is a Howe datum then we let $\pi (\Phi)$ denote the associated tame supercuspidal representation.

\subsection{Embeddings of $E^\times$ in $\GL_n(F)$}
\label{sec:embeddings}
One can associate an $F$-embedding $E\hookrightarrow M(n,F)$ to any $F$-basis $e_1,\dots ,e_n$ of $E$ as follows.
 When $x\in E$ let $$\v{x} =
\left( \begin{array}{c}
x_1\cr 
\vdots\cr
x_n
\end{array}\right),$$ where $x = x_1e_1+\cdots +
x_ne_n$ and $x_1,\dots ,x_n\in F$.  Thus
$x\mapsto \v{x}$ is the standard linear
isomorphism $E\cong F^n$ associated to our
choice of basis.
Multiplication by $x$ is an $F$-linear
transformation of $E$ and hence defines a
matrix $\m{x}\in \g$.  So $x\mapsto \m{x}$
is the regular representation associated
to our basis.  Given $x,x'\in E$, we have
the relations $\m{x}\ \v{x'} = \v{xx'}$ and
$\m{x}\ \m{x'} = \m{xx'}$.

The embedding $x\mapsto\m{x}$ for $x\in E^\times$ is the restriction of an embedding of
algebraic groups $R_{E/F}\GL_1\rightarrow \GL_n$ which we now describe.
Note that $R_{E/F}\GL_1$ is
isomorphic over $\overline F$ to a direct product $\prod_{\sigma\in\Sigma}\GL_1$ indexed
by the set $\Sigma$ of $F$-embeddings
of $E$ in $\overline F$.
Fix an ordering $\sigma_1, \ldots
,\sigma_n$ of the $n$ elements of $\Sigma$.  Let $\iota':
R_{E/F}\GL_1\rightarrow \GL_n$ be the corresponding embedding
$$(x_1,\ldots ,x_n)\mapsto{\rm diag}(x_1,\ldots , x_n).$$
Let
$\mu\in\GL_n(E)$ be the matrix 
$$(\mu_{ij}) = (\sigma_i (e_j)).$$
Let $\iota:R_{E/F}\GL_1\rightarrow \GL_n$ be the embedding
$$\iota = \Int(\mu)^{-1}\circ\iota'.$$
Then $\iota$ is defined over $F$ and
$$\iota(x) = \m{x}\qquad\mbox{for $x\in (R_{E/F}\GL_1)(F) =
  E^\times$.}$$
To see that this equality holds, observe that for $x\in E^\times$, the eigenvalues of $\m x$ are precisely the $\sigma(x_i)$, and the corresponding eigenvectors are the columns of $\mu^{-1}$.  Thus $\mu\m x\mu^{-1} = {\rm diag}(\sigma(x_1),\ldots ,\sigma(x_n))$.   But $(\sigma_1(x),\ldots ,\sigma_n(x))$ is precisely the element of $(R_{E/F}\GL_1)(F)$ that corresponds to $x\in E^\times$.

\begin{lemma}
\label{Fembed}
Given an $F$-embedding $\iota :E\to M(n,F)$, there exists an $F$-basis
  $e_1,\dots , e_n$ of $E$ such that $\iota$ is identical to the
  embedding associated as above to $e_1,\dots ,e_n$.  The same is true for any $F$-embedding $\iota
  :R_{E/F}\GL_1 \rightarrow \GL_n$.
\end{lemma}
We note that in both parts of this lemma, the ordered frame $(Fe_1,\dots, Fe_n)$ is uniquely determined by $\iota$, while the unordered frame $\{ Fe_1,\dots, Fe_n\}$ is uniquely determined by the image of $\iota$.
 
 \begin{proof}
Fix an arbitrary $F$-basis $e'_1,\dots , e'_n$ of $E$.  Let $x\mapsto \m{x}$ be the embedding associated to this basis.  To prove the first statement., note that by the Sk\"olem-Noether Theorem, there exists $g\in G$ such that $g\iota (x)g^{-1} = \m{x}$, for all $x\in E$.  Let $e_j = \sum_i g_{ij} e'_i$.  It is routine to verify that $\iota$ is the embedding associated to $e_1,\dots , e_n$.

  To prove the second statement, let
  $\beta: R_{E/F}\GL_1\rightarrow \GL_n$ be the $F$-embedding associated to the basis $e'_1,\dots , e'_n$.
  Let $\bT = \im \iota$ and $\bT' = \im
  \beta$.  Let $t$ be a regular element of $T$, and let $t' =
  (\beta\circ\iota^{-1})(t)$.  Then $t$ and $t'$ have the same
  eigenvalues since they correspond to the same element of $E^\times$.
  Thus $t' = gtg^{-1}$ for some element $g\in G$.  Since $\bT$ and
  $\bT'$ are the respective centralizers of $t$ and $t'$ in $\bG$, it
  follows that $\bT' = {\rm Int}(g)(\bT)$.  Moreover, the automorphism
  ${\rm Int}(g^{-1})\circ\beta\circ \iota^{-1}$ of $\bT$ fixes the
  regular element $t$ and hence must be the identity map.  Thus $\iota
  = {\rm Int}(g^{-1})\circ \beta $, and since $\beta$ is associated to
  the basis $e'_1,\dots , e'_n$, it follows that $\iota$ is associated
  to another basis $e_1,\dots , e_n$, whose relationship to the original basis is given by the the transition matrix $g$.
\end{proof}

It follows from the preceding lemma that any $F$-embedding $\iota: R_{E/F}\GL_1\rightarrow \GL_n$ gives rise to a unique $F$-embedding $E\rightarrow M(n,F)$ that agrees with $\iota$ on $E^\times = (R_{E/F}\GL_1)(F)$.  Moreover, every such embedding $E\rightarrow M(n,F)$ arises in this way.  In the following, we will typically use the same symbol to denote both of these associated embeddings.

\subsection{Attaching a cuspidal $G$-datum
to a Howe datum}
\label{sec:attaching}

Fix a Howe datum $\Phi = (\varphi, E, \{ \varphi_i\}, \{ E_i\}, \iota : E\hookrightarrow M(n,F))$.  The purpose of this section is to associate a cuspidal $G$-datum $\Psi = (\vec\bG,y,\rho, \vec\phi)$ to $\Phi$.   

Recall from Section~\ref{sec:embeddings} that $\iota$ determines a unique $F$-embedding (which we also denote by $\iota$) of $R_{E/F}\GL_1$ into $\GL_n$.   Let $\bT$ be the image of $\iota$ in $\GL_n$.  Then $\bT$ is an elliptic maximal $F$-torus of $\GL_n$ and $T= \iota (E^\times )$.  

Given an element $x$ of $R_{E/F}\GL_1$ and an $F$-embedding
$\sigma\in\Sigma$, let $x_\sigma$ denote the $\sigma$-component of
$x$.  For $i\in \{ 0,\dots ,d\}$, the torus $R_{E_i/F}\GL_1$ embeds naturally in $R_{E/F}\GL_1$ as the
subgroup consisting of elements $x$ such that
$$x_\sigma = x_\tau\ \ \mbox{if}\ \ \sigma|_{E_i} = \tau|_{E_i}.$$
Let $\bZ^i$ be the image of $R_{E_i/F}\GL_1$ under $\iota$.  Then $\bZ^i$ is an $F$-subtorus of $\bT$
and $Z^i = \iota( E_i^\times)$.
Let $\bG^i$ be the centralizer of $\bZ^i$ in $\bG$.  
Then our desired tamely ramified twisted Levi sequence is  $\vec\bG = (\bG^0,\dots ,\bG^d)$.

\begin{lemma} For all $i\in \{ 0,\dots ,d\}$,  the group $\bG^i$ is $F$-isomorphic to the group $R_{E_i/F} \GL_{n_i}$, where $n_i=n[E_i:F]^{-1}$
and $R_{E_i/F}$ denotes restriction of scalars from $E_i$ to $F$.  
Thus $G^i \cong \GL_{n_i}(E_i)$.
\end{lemma}

\begin{proof}
We may assume the $F$-embeddings $\sigma_1,\ldots ,\sigma_n$ of $E$ in $\overline
F$ are are arranged in $[E_i:F]$ consecutive strings of size $n_i$ such that the
embeddings in each string have the same restriction to $E_i$.  Then the
image of $R_{E_i/F}\GL_1$ under the map $\iota'$
defined above consists of diagonal matrices such that entries
corresponding to elements of a common string are equal.  The
centralizer of $\iota'(R_{E_i/F}\GL_1)$ in $\bG$ is therefore the standard
block-diagonal Levi subgroup $\bM = \GL_{n_i}\times\cdots\times\GL_{n_i}$.
By Lemma~\ref{Fembed}, $\iota$ is associated as above to an
$F$-basis of $e_1,\ldots ,e_n$ of $E$.  Thus, according to the above discussion, $\iota$
must equal $\Int(\mu)^{-1}\circ\iota'$,  where $\mu = (\mu_{ij}) =
(\sigma_i(e_j))$.  It follows that $\bZ_i$ has centralizer $\bG^i = \mu\bM\mu^{-1}$.  Thus, over $\overline
F$, 
$$\bG^i\cong\bM\cong R_{E^i/F}\GL_{n_i}.$$
Moreover, it is readily checked that the action of $\gal (\overline F/F)$
on $\bG^i$ is such that $\bG^i$ and $R_{E^i/F}\GL_{n_i}$ are
isomorphic over $F$.
\end{proof}

Given $i\in \{\, 0,\dots,d\,\}$, there is a homomorphism $\det_i :G^i\to E_i^\times$ that corresponds to the determinant on $\GL_{n_i}(E_i)$ and is independent of the choice of isomorphism $G^i\cong \GL_{n_i}(E_i)$.
We let
$$\phi_i=\varphi_i\circ\det\nolimits_i$$ and $\vec\phi = (\phi_0,\dots ,\phi_d)$.

Let $E'$ be a normal closure of $E/F$.  The space $$A(\bG,\bT,F)= A(\bG,\bT,E')^{\gal (E'/F)}$$  is  1-dimensional.  
The point $y$ in our datum $\Psi$ is chosen to be an arbitrary point in $A(\bG,\bT,F)$.  The corresponding point $[y]$ in the reduced building is uniquely determined by $\bT$.

In the toral case,  we let $\rho$ be the trivial representation
of $G^0=E^\times$.

Now suppose we  are in the   nontoral case. 
Let $q_0$ be the cardinality of the residue class field
of $E_0$. 
Then $G^0_{y,0}$ is conjugate to 
$\GL_{n_0}(\gO_{E_0})$ and $G_{y,0:0^+}^0
\cong \GL_{n_0}(\f_{E_0})$. 
The quasicharacter $\varphi_{-1}$ is
not fixed by any nontrivial element of $\gal(E/E_0)$, since $\varphi_{-1}$ is $E_0$-admissible and  $f(\varphi_{-1})=1$.
The restriction $\varphi_{-1}\,|\, \gO_E^\times$
factors to a character $\lambda$ of $\f_E^\times$
that is in general position in the sense that 
it is not fixed by any nontrivial element
of $\gal(\f_E/\f_{E_0})$.

The construction of Deligne and Lusztig yields a bijection  between the set of equivalence
classes of irreducible cuspidal representations
of $\GL_{n_0}(\f_{E_0})$ and
the $\gal(\f_E/\f_{E_0})$-orbits
of characters of $\f_E^\times$ that are in
general position.
In particular, the above character $\lambda$  determines
an equivalence class $R_\lambda$ of irreducible cuspidal
representations of $\GL_{n_0}(\f_{E_0})$.
Let $\rho^\circ$ be an irreducible representation
of $G_{y,0}^0$ whose restriction to $G_{y,0^+}^0$
is a multiple of the trivial representation
and assume that $\rho^\circ$ factors to an irreducible cuspidal representation of
$G_{y,0:0^+}^0$ belonging to $R_\lambda$.
Note that $$K^0 = G_{[y]}^0=\m{E_0^\times} G_{y,0}^0
\cong \langle \m{\varpi_{E_0}}\rangle \times
G_{y,0}^0,$$ for any choice of prime element
$\varpi_{E_0}$ in $E_0$. Let $\rho$ be the representation
of $K^0$ that restricts to $\rho^\circ$ on 
$G_{y,0}^0$, and such that $\rho(\m{\varpi_{E_0}})$
is equal to $\varphi_{-1}(\varpi_{E_0})$ times the
identity operator on the space of $\rho^\circ$.

  Let $r_i$ be the depth of $\phi_i$, 
for $i\in \{\,0,\dots,d-1\,\}$.   Then $$r_i=\frac{f_i-1}{e}\mbox{, where }f_i=f(\varphi_i\circ N_{E/E_i}),$$
where $e$ is the ramification degree of $E$ over
$F$. 

We have now fully constructed our desired cuspidal $G$-datum  $\Psi$. Note that we had some limited freedom in choosing $y$ and, when $E\not=E_0$, we could vary
the choice of $\rho$ so long as  $\rho\,|\, G_{y,0}^0$ factors to
an element of $R_\lambda$.

\section{Orthogonal involutions}

For a symmetric matrix $\nu\in G$, let $\theta_\nu$ be the $F$-involution of $\GL_n$ given by $x\mapsto \nu^{-1}\cdot {}^t x^{-1}\cdot\nu$.  Here $X\mapsto {}^t X$ denotes the usual transpose on $n\times n$ matrices.

We will refer to involutions of $G$ of the form $\theta_\nu$ as \emph{orthogonal involutions}.

\subsection{Restrictions of orthogonal involutions}

In this section, we prove that if if $\theta$ is an orthogonal involution of $G$ and if $\vec\bG = (\bG^0,\dots ,\bG^d)$ is a tamely ramified twisted Levi sequence associated to $G$ such that $\theta (\vec\bG) = \vec\bG$, then $\theta$ restricts to an orthogonal involution of each group $G^i$.  Implicit in this statement is that each $G^i$ is isomorphic to a general linear group, however, the choice of isomorphism $G^i \cong \GL(n_i ,E_i)$ is irrelevant for our result.

It is important to stress that a given element $g$ of some $G^i$ has a transpose with respect to $G = \GL(n,F)$ and another transpose with respect to $G^i \cong \GL(n_i,E_i)$ (once a specific isomorphism is chosen).  Therefore, if $\nu\in G$ is symmetric (as an element of $G$) and if $\nu$ lies in $G^i$ then it is not necessarily the case that $\theta_\nu$ restricts to an orthogonal involution of $G^i$.

Our assertion about restrictions of orthogonal involutions amounts to showing that an orthogonal involution of $G = G^d$ restricts to an orthogonal involution of $G^{d-1}\cong \GL(n_{d-1},E_{d-1})$, since once this is established one can apply the same result to $G^{d-1}$ and $G^{d-2}$, and so forth, until one deduces that $\theta$ restricts to an orthogonal involution of $G^0$.  For notational simplicity, we write $\bG'$ instead of $\bG^{d-1}$ in this section.

\begin{proposition}\label{orthinv} If $\theta$ is an orthogonal involution of $G$ such that $\theta (G') = G'$ then $\theta$ restricts to an orthogonal involution of $G'$.
\end{proposition}

Fix an orthogonal involution $\theta$ of $G$ such that $\theta (G) = G$.
Fix a symmetric matrix $\nu$ in $G$ such that $\theta (g) = \nu^{-1}\cdot {}^tg^{-1}\cdot \nu$, for all $g\in G$.

We can (and do) fix an isomorphism $G'\cong \GL(n', E')$, where $E'$ is an intermediate field of $E/F$ and $n' = [E:E']$.  We observe that our proof of Proposition \ref{orthinv} uses the fact that $n'= n_{d-1} = [E:E_{d-1}]$ is odd, but otherwise it does not use our assumption that $n$ is odd.

Let $X\mapsto {}^\tau X$ be the transpose on $M(n',E')$.  The proposition we are considering asserts that there exists $\xi \in G'$ such that ${}^\tau \xi = \xi$ and $\theta (g) = \xi^{-1} \cdot {}^\tau g^{-1}\cdot \xi$, for all $g\in G'$.

\begin{lemma}
Under the assumptions of Proposition~\ref{orthinv}, $X\mapsto \nu^{-1}\cdot {}^t X\cdot \nu$ preserves $M(n',E')$.
\end{lemma}

\begin{proof}
Denote the anti-automorphism $X\mapsto \nu^{-1}\cdot {}^t X\cdot \nu$ of $M(n,F)$ by $\alpha$.  Then $\alpha (G') = G'$.  It is easy to choose an $E'$-basis of $M(n',E')$ consisting of elements of $G'$.  But $\alpha$ maps such a basis to another such such basis.  Therefore, $\alpha$ preserves $M(n',E')$.
\end{proof}

\noindent {\it Proof of Proposition \ref{orthinv}.} Define an $E'$-algebra automorphism of $M(n', E')$ by $$\beta (X) = \nu^{-1}\left( {}^t \left( {}^\tau X\right)\right) \nu.$$  By the Sk\"olem-Noether Theorem, there exists $\xi\in G'$ such that $\beta (X) = \xi^{-1} X\xi$, for all $X\in M(n',E')$.

Taking $X= {}^\tau g^{-1}$, we obtain $\theta (g)  = \nu^{-1} \left( {}^tg^{-1}\right) \nu = \xi^{-1} \left( {}^\tau g^{-1}\right) \xi$.
Therefore, $g = \theta (\theta (g)) = \xi^{-1}({}^\tau \xi)g ({}^\tau \xi^{-1})\xi$.  This says that the element $z= \xi^{-1}({}^\tau \xi)$  lies in the center $Z'$ of $G'$.
It now suffices to show that $z=1$.

We note that $\xi z = {}^\tau \xi$ and thus $z= ({}^\tau \xi) \xi^{-1}$.  Therefore, $z^{-1} = {}^\tau z^{-1} = \xi^{-1}\cdot {}^\tau \xi = z$.  Thus, $z= \pm 1$.  But we have (identifying $Z'$ with $(E')^\times$) $1 = \det_{G'} \left(\left( {}^\tau \xi\right) \xi^{-1}\right) = \det_{G'} (z) = z^{n'} =z$. {}
\strut\hfill$\square$

%\subsubsection{Alternate argument}
%Since $\bG'$ is an $E$-twisted Levi subgroup of $\bG$ isomorphic to $R_{E'/F}\GL(n')$, there exists $g\in \bG(E)$ such that $\bL = {}^g \bG'$ is the standard Levi subgroup of $\bG$ consisting of $[E':F]$ blocks of size $n'$.  Then $g\cdot\theta$ is an orthogonal involution of $\bG (E)$ that stabilizes $\bL$  associated to the symmetric matrix $\dot{\nu} = g\theta(g)\nu^{-1}\in \bG (E)$.  {\it In addition, since $g\cdot \theta$ stabilizes $\bL$, $\dot\nu$ must be block diagonal of the same type as $\bL$.}\footnote{I don't understand why there is not a permutation matrix involved.  --JH}  It follows that each block in $\dot\nu$ is symmetric.  Thus $\bL^{g\cdot\theta}$ is a product of $[E':F]$ orthogonal groups in $n'$ variables, each of which is defined over $E$.  Correspondingly, $(\bG')^\theta$ must be $E$-isomorphic to such a product, say $\bH_1\times\cdots\times \bH_{[E':F]}$.  It follows from the action of $\gal (\bar F/F)$ on $\bG'\cong R_{E'/F}\GL(n')$ that $(\bG')^\theta$ is isomorphic to the group $R_{E'/F}\bH_1$.  Thus $(\bG')^\theta (F)$ is isomorphic to the set of $E$-points of an orthogonal group over defined over $E$.

\subsection{$\theta$-split embeddings of $E^\times$}

In this section, we work in the following generality: $E/F$ is a finite separable extension of degree $n$ of arbitrary fields, where $n$ is an integer (possibly even) greater than 1.
 The separability assumption is required because we need to know that the trace $\tr_{E/F}$ is not identically zero and, in addition, we need $E/F$ to have primitive elements.

\subsubsection{Parametrization of $\theta$-split embeddings}

Fix an $F$-basis $e_1,\dots ,e_n$  of $E$.  We refer the reader to Section~\S\ref{sec:embeddings} for the notation $\v x$ and $\m x$ (for $x\in E^\times$) defined with respect to this basis.

Given $x,x'\in E$ and $a\in E^\times$, then
$$\langle x,x'\rangle_a = \tr_{E/F} (axx')$$
defines a symmetric $F$-bilinear form on $E$.
The matrix of this inner product is the symmetric matrix $\nu^a = (\nu^a_{ij})$ in $G$ defined by
 $$\nu^a_{ij} = \tr_{E/F}( ae_ie_j).$$
Thus $$\langle x,x'\rangle_a= \tr_{E/F} (axx') = {}^t \v{x}\cdot \nu^a\cdot \v{x'},$$ for all $x,x'\in E$.

\begin{lemma} The inner product $\langle \ ,\ \rangle_a$ is nondegenerate or, equivalently, $\nu^a$ is invertible.
\end{lemma}

\begin{proof}Assume $\nu^a$ is not invertible.  Then $0$ is an eigenvalue.  Choose $c \in E^\times$ so that $\v{c}$ is an associated eigenvector.
Then $$\sum_{j=1}^n c_j \tr_{E/F}(ae_ie_j) =0$$ for all $i$.  Equivalently, $$\tr_{E/F} (ace_i) =0$$ for all $i$.  But $\{ ace_1,\dots , ace_n\}$ is an $F$-basis of $E$.  Therefore, we deduce that $\tr_{E/F}$ is identically zero.  This contradicts the assumption that $E/F$ is separable.
\end{proof}

We now consider the mapping  $$E^\times \to \{\mbox{nondegenerate symmetric $F$-bilinear forms on $E$}\}$$ given by $a\mapsto \langle \ ,\ \rangle_a$, as well as variations on this mapping.
If an $F$-basis of $E$ has been fixed then $\langle\ ,\  \rangle_a$ determines a symmetric matrix $\nu^a$ in $G$.  
So we have a map $$E^\times \to \cS := \{\mbox{symmetric matrices in $G$}\}.$$

There is a natural action of $G$ on $\cS$: for $g\in G$ and $\nu\in\cS$, define $g\cdot\nu = g\nu\, {}^t\! g$.  We will say that two elements of $\cS$ in the same $G$-orbit are \textit{similar}.

Changing the basis chosen above has the effect of replacing $\nu^a$ by another matrix that is similar to $\nu^a$, so we obtain a 
$$E^\times \to \cO_G(\cS) := \{\mbox{$G$-orbits in $\cS$}\}.$$

If $a,b\in E^\times$ then 
$$\nu^{ab} = {}^t\m{b}\ \nu^a = \nu^a\ \m{b},$$ from which it follows 
that $$\nu^{ab^2} = {}^t\m{b}\ \nu^a \ \m{b}$$ and thus $\nu^{ab^2}$ is similar to $\nu^a$.  Therefore, our map $E^\times\to \cO_G(\cS)$ gives rise to a canonical map
$$X_E  \to \cO_G (\cS),$$ where
$$X_E := E^\times /(E^\times)^2.$$

Each $\nu\in \cS$ determines an involution $\theta_\nu$ of $G$ by
$$\theta_\nu (g) = \nu^{-1}\cdot {}^t g^{-1}\cdot \nu.$$    For simplicity, we write $\theta^a$ instead of $\theta_{\nu^a}$.  

Let $\bT$ be a torus in $\bG$.  For simplicity, we will often refer to the group $T = \bT(F)$ as a \textit{torus} in $G$.  If $\theta$ is an involution of $G$, such a torus $T$ is said to be {\it $\theta$-split} if all of its elements $g$ satisfy $\theta (g) = g^{-1}$.  Since $T$ is dense in $\bT$ with respect to the Zariski topology, this is equivalent to the condition $\theta (g) = g^{-1}$ for all $g\in\bT$, and we will also say that that $\bT$ is \emph{$\theta$-split} in this case.

Now fix $T= \m{E^\times}$.  Then $T$ is $\theta^a$-split, according to the calculation:
$$\theta^a (\m{x}^{-1}) = (\nu^a)^{-1}\cdot {}^t\m{x}\cdot \nu^a = (\nu^a)^{-1}\cdot \nu^a\cdot \m{x} = \m{x}.$$

\begin{lemma}
\label{ecrossmodfcross}
The map $a\mapsto \theta^a$ gives a bijection between $E^\times/F^\times$ and the set of orthogonal involutions $\theta$ of $G$ for which $T= \m{E^\times}$ is $\theta$-split.
\end{lemma}

\begin{proof}
We first consider injectivity.  Suppose $a_1,a_2\in E^\times$.  Then the condition $\theta^{a_1} = \theta^{a_2}$ is equivalent to the condition that $\nu^{a_1}$ and $\nu^{a_2}$ (or the associated inner products) are scalar multiples of each other.  It is easy to see that this is equivalent to the existence of $z\in F^\times$ such that $\tr_{E/F} ((a_1-za_2)x)=0$ for all $x\in E$.  Separability of $E/F$ then says that this is equivalent to $a_1 = za_2$, which proves injectivity.

We now consider surjectivity.  Suppose $\theta$ is an orthogonal involution  such that $T$ is $\theta$-split.  Choose a symmetric matrix $\nu\in G$ such that $\theta (g) = \nu^{-1}\cdot {}^t g^{-1}\cdot \nu$, for all $g\in G$.  (Up to scalar multiples, $\nu$ is uniquely determined by $\theta$.)  
We need to show that there exists $a\in E^\times$ such that $\nu = \nu^a$.

For $x,y\in E$, define $\langle x,y\rangle_\nu = {}^t\v{x}\cdot \nu\cdot \v{y}$ and let $\phi_\nu\in \Hom_F(E,F)$ be defined by $\phi_\nu (z) = \langle 1,z\rangle_\nu$.
We now observe that every nonzero element of $\Hom_F (E,F)$ is associated to an element of $E^\times$ in the following manner.  Define an $F$-linear map $E\to \Hom_F(E,F)$ by mapping $a\in E$ to $\tr_{E/F}\circ \mu_a$, where $\mu_a :E\to E$ is given by $\mu_a(x) = ax$.  This map is clearly injective.  Therefore, it defines an $F$-linear isomorphism $E\cong \Hom_F(E,F)$ since $E$ and $\Hom_F(E,F)$ both have $F$-dimension $[E:F]$.  This implies that there exists $a\in E^\times$ such that 
$\phi_\nu (x) = \tr_{E/F}(ax)$, for all $x\in E$.

Now suppose $x,y\in E$.  Note that $\theta (\m{x}^{-1}) = \m{x}$, from which it follows that $\nu\cdot  \m{x} =  {}^t\m{x}\cdot \nu$.  Therefore, we have $\langle x,y\rangle_\nu = {}^t\v{x}\cdot \nu\cdot \v{y} = {}^t\v{1}\cdot {}^t\m{x}\cdot\nu\cdot \v{y}= {}^t\v{1}\cdot \nu\cdot \m{x}\v{y} = \langle 1,xy\rangle_\nu = \phi_\nu (xy)=\tr_{E/F} (axy) = \langle x,y\rangle_a$.
We now deduce that $\nu = \nu^a$ which completes the proof.
\end{proof}

We now observe that
$$\theta^{ab^{-2}} = \m{b}\cdot \theta^a$$ and interpret this as an equivariance property of the map $a\mapsto \theta^a$.  More precisely, let $E^\times$ act on $E^\times/F^\times$ by $b\cdot (aF^\times) = ab^{-2}F^\times$, and let $T$ act on the set of involutions of $G$ by restricting the usual action of $G$ on involutions.  Then the mapping $a\mapsto \theta^a$ becomes equivariant with respect to $T = E^\times$, where we identify $E^\times$ with $T$ via $x\mapsto \m{x}$.   This yields:
 
 \begin{corollary}\label{corcorr}
Let $T= \m{E^\times}$.  The map $a\mapsto \theta^a$ gives a bijection from $E^\times/ (E^\times)^2F^\times$
to the set of $T$-orbits of orthogonal involutions $\theta$ of $G$ for which $T$ is $\theta$-split.
\end{corollary}

\subsubsection{$Y_{E/F}$}\label{sec:yEFsec}

We have just shown that $a\mapsto \theta^a$ gives a bijection 
 $$\mu_{E/F} : Y_{E/F} \to \cO^T,$$ where
$Y_{E/F} = E^\times /(E^\times)^2F^\times$ and 
$\cO^T$
is the set of $T$-orbits of orthogonal involutions $\theta$  such that $T$ is $\theta$-split.    Now let $y_{E/F}$ denote the cardinality of $Y_{E/F}$.

\begin{lemma}\label{yEF}
 $y_{E/F}-1$ is the number  of quadratic extensions of $F$ contained in $E$.  In particular, $y_{E/F} =1$ if $n$ is odd and $y_{E/F}=2$ if $n=2$.
\end{lemma}

\begin{proof} We start by noting that the nontrivial elements of $X_{E} = E^\times /(E^\times)^2$ represent quadratic extensions of $E$.  So the elements of $Y_{E/F}$ may be viewed as quadratic extensions of $E$ modulo those of the form $EF'$, where $F'$ is a quadratic extension of $F$.

Now assume $E/F$ is a degree $n$ extension of $p$-adic fields of characteristic zero with $p\ne 2$.   Here is another interpretation of $Y_{E/F}$ and $y_{E/F}$.
We rewrite $Y_{E/F}$ as
$$\frac{E^\times/ (F^\times\cap (E^\times)^2) }{\left( (E^\times)^2/ (F^\times\cap (E^\times)^2)\right) \times \left( F^\times/ (F^\times\cap (E^\times)^2)\right)}$$ and then see that
$$y_{E/F} = \frac{ |E^\times/ (F^\times\cap (E^\times)^2) |}{ \left| (E^\times)^2/ (F^\times\cap (E^\times)^2)\right|\cdot \left| F^\times/ (F^\times\cap (E^\times)^2)\right|}.$$ This implies 
$$y_{E/F} = \frac{4}{\left| F^\times/ (F^\times\cap (E^\times)^2)\right|}.$$
Now the nontrivial elements of $(F^\times\cap (E^\times)^2)/(F^\times)^2$ are in bijective correspondence with the quadratic extensions of $F$ that are contained in $E$ and we have
\begin{eqnarray*}
\left| F^\times/ (F^\times\cap (E^\times)^2)\right|
&=& \frac{\left| F^\times/ (F^\times)^2\right|}{\left|  (F^\times\cap (E^\times)^2)/(F^\times)^2\right|}\\
&=& \frac{4}{\left|  (F^\times\cap (E^\times)^2)/(F^\times)^2\right|.}
\end{eqnarray*}
Our claim follows.\end{proof}

\subsubsection{Split orthogonal involutions}

Let  $$J= J_n =\begin{pmatrix}
& &1\\
&\raisebox{-.1ex}{.} \cdot \raisebox{1.2ex}{.}&\\
1&&
\end{pmatrix}.$$  In this section, we consider the orthogonal involution $\theta_J$ and its $G$-orbit $\Theta_J$.  We assume throughout that $E/F$ is a degree $n$ tamely ramified extension of characteristic-zero $p$-adic fields.
 Note that if $\theta\in \Theta_J$ then $G^\theta$ is a split orthogonal group.  We will prove:

\begin{proposition}\label{sporth}
Given an embedding of $E^\times$ in $G$ with image $T$ then there exists $\theta\in \Theta_J$ such that $T$ is $\theta$-split.  
Consequently, $\Theta_J$ must contain a $T$-orbit that lies in $\cO^T$.
Given $\theta\in \Theta_J$ there exists an embedding of $E^\times$ in $G$ whose image $T$ is $\theta$-split.\end{proposition}

Our approach to the proof of Proposition \ref{sporth} involves the characterization of $G$-orbits in $\cS$ using discriminants and Hasse invariants.  
Since there is some inconsistency in the literature regarding the use of the terms ``discriminant'' and ``Hasse invariant,'' we start by defining these terms.

If $s\in \cS$ then the {\it discriminant of $s$}, which we denote by $\disc s$, 
is the class of $\det s$ in $X_F = F^\times /(F^\times)^2$.  Another important notion of discriminant is the notion of the {\it signed discriminant of $s$} which is the class of $(-1)^{n(n-1)/2}\det s$ in $X_F$.  To explain the power of $-1$ in the latter definition, we recall the definition of the Witt group of $F$.  Consider the semigroup consisting of the equivalence classes on nondegenerate finite-dimensional quadratic spaces over $F$ with respect to the direct sum operation.  The quotient of the latter semigroup with the subsemigroup generated by the hyperbolic planes is a group of order 16 called the
 {\it Witt group of $F$}.  The elements of the Witt group are naturally identified with the equivalence classes of finite anisotropic quadratic spaces.  The element in the Witt group associated to any finite sum of hyperbolic planes is just the identity element, and we observe that the signed discriminant of any such quadratic space is trivial.  Thus the signed discriminant has the favorable property that it is a Witt group invariant, whereas the ordinary discriminant is not.
The appearance of the factor $(-1)^{n(n-1)/2}$ at various points in our discussion below can be interpreted to some degree via the Witt group.  We also note that $$\det J_n = (-1)^{n(n-1)/2}.$$
 
If $A$ is a symmetric matrix in $\GL(m,F)$, $m\in {\mathbb N}$,  then we define the {\it Hasse invariant}
 of
$A$  by $${\rm Hasse}(A)=\prod_{i\le j}  (a_i,a_j),$$
where
$\hbox{diag}(a_1,\dots ,a_m)$ is a diagonal matrix in the  $G$-orbit of $A$ and $(\ ,\ )$ is the Hilbert symbol
$$(a,b)=
    \begin{cases}1,&\mbox{if }z^2=ax^2+by^2\mbox{ has a solution }(x,y,z)\in F^3-\{0\};\\
    -1,&\mbox{otherwise.}\end{cases}$$
    
The following classical result is Theorem 63.20 \cite{O}:

\begin{lemma}\label{classicallemma} The $G$-orbits in $\cS$ are characterized by the discriminant and Hasse invariant.  There are eight possibilities for the pair $({\rm disc}(\nu),{\rm Hasse}(\nu))$.  When $n>2$ each of these possibilities corresponds to a different $G$-orbit in $\cS$ and these eight orbits give all the $G$-orbits in $S$.  When $n=2$, there are only seven orbits since it is impossible to have both ${\rm disc}(\nu)=-1$ and ${\rm Hasse}(\nu)=-1$.
\end{lemma}

Note that the Hasse invariant is often defined as a product  over $i<j$,
instead of $i\le j$.   We will let ${\rm Hasse}_0(A)$ denote the latter
version of the  Hasse invariant. Though these two definitions are not
equivalent, either may be used to classify quadratic forms. The discrepancy
between these two definitions is the product
\begin{eqnarray*}\prod_{i=1}^m (a_i,a_i)&=&\prod_{i=1}^m (a_i, -1) = ({\rm disc}(A),-1)\\
&=&
\begin{cases}
    1,&\mbox{if $\prod a_i$ is a sum of two squares,}\\
-1,&\mbox{otherwise.}
    \end{cases}\\
&=&
\begin{cases}
    -1,&\mbox{if $-1\notin (F^\times)^2$ and ${\rm disc}(A)$ has odd valuation},\\
1,&\mbox{otherwise.}
    \end{cases}\\
\end{eqnarray*}

\begin{lemma}\label{hasselemma}
 If $\nu\in \GL_n (\gO_F)\cap \cS$  then $\Hasse (\nu) = \Hasse_0 (\nu) =1$.\end{lemma}

\begin{proof}
 Suppose $a,b\in \gO^\times$.
 Since every quadratic form of dimension 3 is isotropic over the finite field, we see that we may choose $x,y,z\in \gO$, not all all in $\gP$, such that $ax^2 +by^2 \equiv z^2\pmod{\gP}$.
 Suppose $x\notin \gP$.  Let $f(X) = aX^2 + by^2-z^2$.  Applying Hensel's Lemma to $f$, we see that we can find $x'\in \gO$ with $x-x'\in \gP$ such that $f(x')=0$.
 So, assuming $x\notin \gP$, we get an isotropic vector for $aX^2+bY^2 -Z^2$.  If $x\in \gP$ we can argue similarly, replacing $x$ by $y$ or $z$. We deduce that $(a,b) =1$ whenever $a,b\in \gO^\times$.

Fix $\nu \in \GL_n(\gO)\cap \cS$.  To complete the proof, it suffices to show that $\nu$ is similar
to a diagonal matrix in $\GL_n (\gO)$.  Let us regard $F^n$ as a (nondegenerate) quadratic space $V_1$ with respect to the symmetric bilinear form associated to $\nu$.

Since $V_1$ is nondegenerate, it contains anisotropic vectors.  We also note that every anisotropic vector is clearly a scalar multiple of an anisotropic vector in $\gO^n$ that is primitive in the sense that its reduction modulo $\gP^n$ is nonzero.

Choose a primitive anisotropic vector $v_1$ in $\gO^n$.  Let $V_2$ be the orthogonal complement of $v_1$.  Then $V_1$ is an orthogonal direct sum of $Fv_1$ and $V_2$.  Thus $V_2$ must be nondegenerate.  So we may choose a primitive anisotropic element in $V_2$.  Continuing in this way, we obtain an orthogonal basis $v_1,\dots , v_n$ consisting of primitive anisotropic vectors.  The matrix $\nu$ is similar to the diagonal matrix $A$ whose $i$th diagonal entry is $a_i = {}^t v_i \nu v_i\in \gO$.  It now suffices to show that $A\in \GL_n (\gO)$.

We may now pass to the residue field $\f$ of $F$.
The image $\bar v_1,\dots , \bar v_n$ in $\f^n$ of $v_1,\dots ,v_n$ is a basis of $\f^n$.  The image $\bar\nu\in M(n,\f)$ of $\nu$ is symmetric and invertible.  Note that $\bar a_i = {}^t \bar v_i \bar\nu \bar v_i$ is the image of $a_i$ in $\f$.  Since $\bar a_i$ is nonzero for all $i$, the diagonal matrix $A$ must lie in $\GL_n (\gO)$.
\end{proof}

We have observed that we have an identity
$$\nu^a = {}^t\m{a}\nu^1 = \nu^1 \m{a}.$$
Taking determinants yields the identity
$$\det (\nu^a) = N_{E/F}(a)\cdot \det (\nu^1).$$ 
It is perhaps of some interest to note that the latter identity can also be deduced from the following standard (at least when $a=1$) result.

\begin{lemma}\label{standardstuff} Let $\sigma_1,\dots,\sigma_n$ be the distinct $F$-embeddings of $E$ into a fixed algebraic closure $\overline{F}$ of $F$.  Let $A$ be the diagonal matrix whose $i$th diagonal entry is $\sigma_i (a)$ and let $B$ be the matrix whose $ij$-th entry is $\sigma_j(e_i)$.  Then $\nu^a = B\cdot A\cdot {}^t B$.
\end{lemma}

\begin{proof} The $ij$-th entry of $\nu^a$ is $$\tr_{E/F}(ae_ie_j) =\sum_{k=1}^n \sigma_k(a)\sigma_k(e_i)\sigma_k(e_j).$$  But this is the same as the $ij$-th entry of $B\cdot A\cdot {}^tB$.  This yields the desired matrix identity.\end{proof}

\noindent The next lemma is also quite well known ({\it cf.}, Proposition 12.1.4 \cite{IR}).

\begin{lemma}  Let $\beta$ be a primitive element for $E/F$ and take $e_1 =1, e_2 =\beta ,e_3 = \beta^2,\dots,e_n = \beta^{n-1}$.  Let $f$ be the minimal polynomial for $\beta$ over $F$.  Then $\det (\nu^1) = (-1)^{n(n-1)/2}N_{E/F}(f'(\beta)).$
\end{lemma}

Let us now examine the discriminant of the $G$-orbit $\cS^a$ in $\cS$ of $\nu^a$.  First we note that, by definition, the discriminant of $\cS^1$ is just the discriminant $\disc (E/F)$ of the extension $E/F$.
Therefore, $$\disc (\cS^a) = N_{E/F}(a)\cdot\disc (E/F).$$
By Lemma~\ref{standardstuff},
$\disc (E/F)$ lies in the image of $(-1)^{n(n-1)/2}N_{E/F}(E^\times)$ in $X_F$.  The same must therefore be true of $\disc (\cS^a)$.  In other words, the signed discriminant of $\cS^a$ lies in the subset $N_{E/F}(E^\times )/(F^\times)^2$ of  $X_F$.

If $n$ is odd then $N_{E/F}$ defines a surjective map from $E^\times$ to $X_F$ since, in fact, the restriction of this map to $F^\times$ is identical to the natural projection $F^\times\to X_F$.
The following result is now immediate.

\begin{lemma}
The matrices in $\nu^1 \, \m{E^\times}= \{ \nu^a:\ a\in E^\times\}$ are symmetric.  The discriminant classes  represented by these elements comprise the image
 of  
 $$(-1)^{n(n-1)/2} N_{E/F}(E^\times)$$
 in $X_F$.  If $n$ is odd then $\disc$ maps $\nu^1\ \m{E^\times}$ onto $X_F$.
 \end{lemma}
 
To show that there exists $a\in E^\times$ such that $\nu^a$ is similar to $J$, it suffices to show that for some $a$ the matrices $\nu^a$ and $J$ have the same discriminant and the same Hasse invariant.  We start with the case in which $E/F$ is unramified, then we settle the totally and tamely ramified case, and finally we combine the latter cases to obtain the desired result for general tamely ramified extensions.

\begin{lemma}
\label{unram_embed}
If $E/F$ is an unramified extension of degree $n$ then there exists an element $a\in E^\times$ and an $F$-basis of $E$ such that the associated matrix $\nu^a$ is identical to $J$.
\end{lemma}

\begin{proof}
When $a\in E^\times$ and $\beta$ is a primitive element for $E/F$, we have
$$\det (\nu^a) = N_{E/F}(af'(\beta))(-1)^{n(n-1)/2}, $$
where $f$ is the minimal polynomial of $\beta$ and $\nu^a$ is defined with respect to the basis $e_1 =1, e_2 =\beta ,e_3 = \beta^2,\dots,e_n = \beta^{n-1}$.  We may choose such a $\beta$ which, in addition, lies in $\gO_E^\times$.
Take $a = f'(\beta)^{-1}$.  
Since $f$ is irreducible modulo $\gP$, the image of $f'(\beta)$ in $\gO_E/\gP_E$ is nonzero, and hence $f'(\beta)$ and $a$ are units.

Since $\tr_{E/F}$ takes $\gO_E$ to $\gO$, $\nu^a$ has entries in $\gO$.  Moreover, since $\det\nu^a = (-1)^{n(n-1)/2}\in\gO^\times$,
it follows that $\nu^a\in\GL_n (\gO)$.  Thus, according to Lemma \ref{hasselemma}, we have $\Hasse (\nu^a) = 1=\Hasse (J)$.  We also have
$$\det (\nu^a) = (-1)^{n(n-1)/2} = 
\det J.$$
Therefore, by Lemma \ref{classicallemma}, $\nu^a$ must be similar to $J$.
Now choose $g\in G$ such that $g\cdot \nu^a \cdot {}^t g = J$.  Define an $F$-basis $e'_1,\dots , e'_n$ of $E$ by $e'_i = \sum_j g_{ij} e_j$.  Then the matrix $\nu^a$ with respect to $e'_1,\dots , e'_n$ is precisely $J$.
\end{proof}

\begin{lemma}
\label{ram_embed}
If $E/F$ is a totally and tamely ramified extension of degree $n$ then there exists an element $a\in E^\times$ and an $F$-basis of $E$ such that the associated matrix $\nu^a$ is identical to $J$.
\end{lemma}

\begin{proof}
As in the proof of Lemma \ref{unram_embed}, we choose a certain
$a\in E^\times$ and a certain primitive element $\beta$ for $E/F$, and we use  the  $F$-basis $e_1 =1, e_2 =\beta ,e_3 = \beta^2,\dots,e_n = \beta^{n-1}$ of $E$.  
In the present case, we take $\beta$ to be an element of $E$ that is an $n$-th root of a prime element $\varpi_F$ in $F$.  (The fact that this is possible follows from Proposition 12 \cite{sL}.)

We observe that
$$\m{\beta} = \begin{pmatrix}
0&\cdots &0&\varpi_F\\
1&\ddots&&0\\
&\ddots &\ddots&\vdots\\
0&&1&0
\end{pmatrix} $$
It is easy to evaluate $\tr_{E/F}(\beta^k)= \tr (\m{\beta}^k)$ for any $k$ and to verify that the trace is zero unless $k$ is a multiple of $n$.
Taking $a= \beta^{1-n}/n$, we obtain 
$\nu^a =J$.
\end{proof}

\begin{proposition}
\label{tame_embed}
If $E/F$ is a tamely ramified extension of degree $n$ then there exists an element $a\in E^\times$ and an $F$-basis of $E$ such that the associated matrix $\nu^a$ is identical to $J$.
\end{proposition}

\begin{proof}
Let $K/F$ be the maximal unramified subextension of $E/F$.  Let $f = [K:F]$ and $e = n/f = [K:E]$.  
We define a tensor product map
$$M(e,F)\times M(f,F)\to M(n,F)$$
by taking $A\otimes B$ to be the $e\times e$ block matrix whose $ij$-th block is $A_{ij}B\in M(f,F)$.

Now suppose $a\in E^\times$ and let $\alpha$ be a $K$-basis $\alpha_1,\dots , \alpha_e$ of $E$.  Define a symmetric matrix $\nu^a_\alpha \in \GL_e (K)$ by $(\nu^a_\alpha)_{ij} = \tr_{E/K} (a\alpha_i\alpha_j)$.
Similarly, suppose $b\in K^\times$ and let $\beta$ be a $F$-basis $\beta_1,\dots , \beta_f$ of $K$.  Define a symmetric matrix $\nu^b_\beta \in \GL_f (F)$ by $(\nu^b_\beta)_{k l} = \tr_{K/F} (b\beta_k\beta_l)$.
Let $\alpha\otimes\beta$ be the $F$-basis of $E$ given by
$$\alpha_1\beta_1,\dots , \alpha_1 \beta_f, \alpha_2 \beta_1,\dots ,\alpha_2\beta_f , \dots ,\alpha_e\beta_1,\dots , \alpha_e \beta_f.$$

Suppose that $\nu^a_\alpha$ has entries in $F$.  Then the tensor product $\nu^a_\alpha\otimes \nu^b_\beta$ is defined and, according to the following calculation, it is identical to $\nu^{ab}_{\alpha\otimes \beta}$:
\begin{eqnarray*}
(\nu^a_\alpha\otimes \nu^b_\beta)_{ijkl} 
&=& (\nu^a_\alpha)_{ij} (\nu^b_\beta )_{kl}\\
&=& \tr_{E/K} (a\alpha_i\alpha_j)\tr_{K/F} (b\beta_k\beta_l)\\
&=&\tr_{K/F} (b\beta_k\beta_l  \tr_{E/K} (a\alpha_i\alpha_j))\\
&=&\tr_{K/F} (  \tr_{E/K} ( b\beta_k\beta_l a\alpha_i\alpha_j))\\
&=&\tr_{E/F} ( a b\alpha_i\alpha_j\beta_k\beta_l )\\
&=&(\nu^{ab}_{\alpha\otimes \beta})_{ijkl} .
\end{eqnarray*}

By Lemma~\ref{unram_embed}, we may choose $a$ and $\alpha$ such that $\nu^a_\alpha = J_e$.
By Lemma~\ref{ram_embed}, we may choose $b$ and $\beta$ such that $\nu^b_\beta = J_f$.  
Then $J_n = J_e\otimes J_f = \nu^a_\alpha \otimes \nu^b_\beta = \nu^{ab}_{\alpha\otimes \beta}$
which proves our claim.
\end{proof}

\medskip\noindent{\it Proof of Proposition  \ref{sporth}.}
Proposition \ref{tame_embed} and Lemma \ref{ecrossmodfcross} imply  that there exists an embedding $x\mapsto \m{x}$ of $E^\times$ in $G$ whose image $T$ is $\theta_J$-split.  If $g\in G$ and $\theta = g\cdot \theta_J$ then $g\mapsto g\m{x} g^{-1}$ defines an embedding of $E^\times$ in $G$ whose image is $\theta$-split.

Now suppose we are given an embedding of $E^\times$ in $G$ and let $T$ denote its image.  
Lemma \ref{Fembed} implies that the embedding must come from an $F$-basis $e_1,\dots , e_n$ of $E$.  Proposition \ref{tame_embed} says that there must exist another basis $e'_1,\dots ,e'_n$ and $a\in E^\times$ such that the associated matrix $\nu^a$ is $J$.
The change-of-basis matrix in $G$ between these bases sends $\theta_J$ to a matrix $\theta\in \Theta_J$ such that 
 $T$ is $\theta$-split.  
Consequently, $\Theta_J$ must contain a $T$-orbit that lies in $\cO^T$ which completes the proof.\hfill$\square$

\subsubsection{Refined results when $n$ is odd}

In this section, we assume  $n$ is odd.  

\begin{lemma}\label{thetastabisthetasplit}
Suppose $L$ is a subring of $M(n,F)$ that is a field extension of $F$ of odd degree.  Assume $\theta$ is an orthogonal involution of $G = \GL_n (F)$ such that $\theta (L^\times)=L^\times$.  Then $\theta (t) = t^{-1}$ for all $t\in L^\times$.
\end{lemma}

\begin{proof}
Choose a symmetric matrix $\nu$ such that $\theta = \theta_\nu$.  Then $\sigma (x) = \nu^{-1}\cdot {}^t x\cdot \nu$ defines an $F$-automorphism of $L$ whose square is the identity map.  Since $\gal (L/F)$ has odd order, $\sigma$ must be identity map on $L$.  This is equivalent to our assertion.
\end{proof}

The latter result shows that every $\theta$-stable torus in $\bG = \GL_n$ must in fact be $\theta$-split.

In the next result, we continue to assume that we have fixed an embedding of $E$ in $M(n,F)$ and we let $T$ denote the image of $E^\times$.

\begin{proposition}\label{refinednoddresult} Assume $n$ is odd.  The $G$-orbit $\Theta_J$  is the unique $G$-orbit of orthogonal involutions of $G$ that contains an involution $\theta$ for which $T$ is $\theta$-stable.  For every such involution $\theta$, the torus $T$ must in fact be $\theta$-split.   The set of all $\theta\in \Theta_J$ such that $T$ is $\theta$-split comprises a single $T$-orbit in $\Theta_J$.  The orthogonal group associated to any element of $\Theta_J$ is a split orthogonal group.  
\end{proposition}

\begin{proof} 
According to Corollary \ref{corcorr} and Lemma \ref{yEF}, the map $\mu_{E/F}: Y_{E/F}\to \cO^T$ of \S \ref{sec:yEFsec} reduces to a bijection between two singleton sets when $n$ is odd.
This says that there is a unique $T$-orbit of orthogonal involutions $\theta$ such that $T$ is $\theta$-split and every involution of the form $\theta^a$, for $a\in E^\times$, lies in this orbit.
The fact that $\theta$-stable tori must be $\theta$-split follows from Lemma \ref{thetastabisthetasplit}.

Proposition \ref{sporth} implies that the latter orbit lies in $\Theta_J$.  It is well known and easily verified that the orthogonal group associated to $\theta_J$ is split.  
The orthogonal groups associated to other elements of $\Theta_J$ are $G$-conjugate to the latter group and hence they must also be split.
\end{proof}

\begin{corollary}
\label{cor:theta-symmetric}
If $n$ is odd and $\theta$ is an orthogonal involution then the following are equivalent:
\begin{itemize}
\item $\theta (\vec\bG) = \vec\bG$,
\item  $Z^0$ is a $\theta$-split torus,
\item $\Psi$ is weakly $\theta$-symmetric.
\end{itemize}
\end{corollary}

\begin{proof}
Assume $Z^0$ is a $\theta$-split torus.  Then $Z^i$ is a $\theta$-split torus for all $i$, since it is a torus and it is contained in $Z^0$.  Since $Z^i$ is $\theta$-split, it is $\theta$-stable and hence so is its stabilizer in $G$.  So $\vec\bG$ is $\theta$-stable.  Conversely, if $\vec\bG$ is $\theta$-stable then $Z_0$ must be $\theta$-stable and hence $\theta$-split by the Lemma \ref{thetastabisthetasplit}.  This establishes the equivalence of the first two conditions.

Now consider the quasicharacter $\phi_i$ in $\vec\phi = (\phi_0,\dots ,\phi_d)$ and assume $\theta (G^i)=G^i$.  Then $\theta$ restricts to an orthogonal involution of $G^i$ with respect to any isomorphism $G^i \cong \GL_{n_i}( E_i)$, according to Proposition \ref{orthinv}.  Thus if $g\in G^i= \theta (G^i)$ then $\det_i(g\theta (g))=1$ so $g\theta(g)$ lies in the commutator subgroup of $G^i$.  This implies that our first and third conditions are equivalent.
\end{proof}

\section{Orthogonal periods}

Suppose $n$ is odd from now on.  Fix a $G$-orbit  $\Theta$ of orthogonal involutions of $G$ and fix a Howe datum $\Phi = (\varphi , E, \{ \varphi_i\}, \{ E_i\},\iota : E\hookrightarrow M(n,F))$ in the sense of Definition \ref{Howedatum}.
Let $\Psi = (\vec\bG ,y,\rho,\vec\phi)$ be a cuspidal $G$-datum that is associated to $\Phi$ as in \S \ref{sec:attaching}.  Recall from Theorem~\ref{thm:K^0_formula} the formula
$$\langle \Theta, \Psi\rangle_G = \sum_{[\theta]\sim[\Psi]} m_{K^0} ([\theta])\ \langle [\theta] , [\Psi]\rangle_{K^0}.$$  
Our objective in this section is to
 compute all of the terms on the right hand side of the latter formula.

Let us briefly sketch our strategy.
From Proposition \ref{refinednoddresult}, it follows that if $\langle \Theta ,\xi\rangle_G$ is nonzero then $\Theta = \Theta_J$.  So let us assume $\Theta = \Theta_J$.
In \S \ref{sec:mKt}, we show that $m_{K^0}([\theta])=1$ for all orbits $[\theta]$ in our examples.  (This is not true for even $n$.)

By definition, if a summand $\langle [\theta],[\Psi]\rangle_{K^0}$ is nonzero then it is equal to the dimension of the space $\Hom_{K^{0,\theta}}(\rho',\eta'_\theta)$.  (See~\S\ref{subsub:simplified}.)
Using our generalized version of Lusztig's results (Theorem \ref{Lusztigwithchi}), we then show that we can assume that our torus $\bT$ in Definition~\ref{cuspidalGdatum} is $\theta$-split.

Next, we use Proposition \ref{refinednoddresult} to identify a particular summand as the only summand that can be nonzero.    In \S \ref{sec:etriv}, we give an explicit formula for $\eta'_\theta$ and show that it is always trivial for our purposes.  Finally, to compute the only relevant summand, we appeal to Lusztig's formula.  (In this case, we do not need to use the generalized form of the formula from~\S\ref{sec:genlusztig}.) 

\subsection{Triviality of $m_{K^0}([\theta])$}
\label{sec:mKt}

\begin{lemma}\label{mKthetaistrivial}  Let $\theta$ be an orthogonal involution of $G$.  Then $\mu (G_\theta) = Z^2$ and $G_\theta = ZG^\theta$.  Consequently, $m_{K^0} ([\theta]) =1$.
\end{lemma}

\begin{proof}
Choose a symmetric matrix $\nu\in G$ such that  $\theta = \theta_\nu$.  The similitude ratio defines a homomorphism $\mu : G_\theta \to Z$.  
We have $$\mu (g) = g\theta(g)^{-1} = g\cdot \nu^{-1}\cdot {}^t g \cdot \nu .$$

If $z\in Z$ then $\mu (z) = z^2$ and thus $\mu (G_\theta)\supset \mu (Z) = Z^2$.  
Let us identify $Z$ with $F^\times$ in the obvious way.  
If $g\in G_\theta$ then $\det \mu (g) = (\det g)^2 \in Z^2$. 
On the other hand, if $z\in Z$ then $\det z = z^n \equiv z$ (mod $Z^2$).
So, $Z^2\supset \mu (G_\theta)\supset Z^2$ and hence $\mu (G_\theta) = Z^2$.

The similitude ratio $\mu$ defines an exact sequence
$$1\to G^\theta \to G_\theta \to \mu (G_\theta)\to 1.$$
This yields an exact sequence
$$1\to G^\theta /\{ \pm 1\}  \to G_\theta /Z \to \mu (G_\theta)/Z^2\to 1,$$
since $Z\subset G_\theta$, $G^\theta\cap Z = \{ \pm 1\}$ and $\mu (Z) = Z^2$.
Hence, we have an isomorphism
$$G_\theta/ ZG^\theta \cong \mu (G_\theta)/Z^2.$$
Since $\mu (G_\theta) = Z^2$, we deduce that $G_\theta = ZG^\theta$.

By Theorem \ref{mKTheta}, we have
$$ m_{K^0}([\theta]) = [G_\theta : (K^0 \cap G_\theta)G^\theta ].$$  
But now $Z\subseteq K^0 \cap G_\theta$ implies
$$[G_\theta : (K^0 \cap G_\theta)G^\theta ]\le   [G_\theta : ZG^\theta ] =1.$$  Therefore, $m_K(\Theta') =1$.
\end{proof}

\subsection{Relevant Involutions}
\label{sec:reduction}
Suppose $\theta$ is an involution of $G$ such that 
$$\langle[\theta],[\Psi]\rangle_{K^0} = \dim {\rm Hom}_{K^{0,\theta}} (\rho' , \eta'_\theta)\neq 0,$$
where $K^0$, $\rho'$ and $\eta'_\theta$ are associated to $\Psi$ as in \cite{HM}.
(Only orbits $[\theta]$ with this property can contribute to the formula for $\langle\Theta,\Psi\rangle_G$ in Theorem~\ref{thm:K^0_formula}.)  Then we must have $[\theta]\sim[\Psi]$, that is, $\theta(K^0) = K^0$, and the character $\phi$ of $K^0$ given by $\phi (g) = \prod_{i=0}^d \phi_i(g)$ restricts trivially to $K^{0,\theta}_+$.  In particular, by Lemma~\ref{lem:K^0}, $\theta$ must stabilize $\bG^0$ and $[y]$.

Let $\bT$ be the $F$-torus in $\bG$ such that $T = \iota (E^\times)$.
Then $\bT$ can be taken to be the torus appearing in Definition~\ref{cuspidalGdatum}. 
We want to show that there always exists a $\theta$-split maximal $F$-torus $\bT'$ of $\bG$ with the properties in
Definition~\ref{cuspidalGdatum}.  By Lemma~\ref{thetastabisthetasplit}, it suffices to show that there is a $\theta$-stable torus $\bT'$ with these properties.

If $\Psi$ is toral, then $\theta$ stabilizes $T=K^0$, and we are done.
So suppose $\Psi$ is nontoral.  
Then there exists an $\f$-group $\mathsf{G}_y^0$ such that $\mathsf{G}_y^0 (\f) = G^0_{y,0:0^+}\cong \GL_{n_0}(\f_{E_0})$.
Let $\mathsf{T}$ be the $\f$-torus in $\mathsf{G}_y^0$ determined by $T$.  (See the Appendix.)  Thus $\mathsf{T} (\f) = T_{0:0^+}$.
 The character $\varphi_{-1} |\gO^\times_E$ projects to a character $\lambda$ of $\mathsf{T}(\f)$.
 
Recall that the Deligne-Lusztig virtual representation $R^\lambda_{\mathsf{T}}$ of $\mathsf{G}_y^0(\f)$ associated to $(\mathsf{T},\lambda)$ is an irreducible cuspidal representation that corresponds to the representation $\rho^\circ$ of $G^0_{y,0}$.
In addition, $\rho$ is the representation of $K^0$ that restricts to $\rho^\circ$ on $G^0_{y,0}$ and acts according to $\rho (\m{\varpi_{E_0}}) = \varphi_{-1} (\varpi_{E_0})$ for any prime element $\varpi_{E_0}$ in $E_0$.

We note that $K^{0,\theta} = G^{0,\theta}_{y,0}$.
The involution $\theta$ determines an involution of $\mathsf{G}_y^0$ that we also denote by $\theta$.  It follows from Proposition 2.12 \cite{HM} that the group of fixed points of $\theta$ in $\mathsf{G}_y^0(\f)$ is the same as the image of $G^{0,\theta}_{y,0}$ in $\mathsf{G}_y^0(\f)$.  Moreover, if we identify $\mathsf{G}_y^0(\f)$ with $\GL_{n_0}(\f_{E_0})$, then there exists an $\f_{E_0}$-involution $\theta_0$ of $\GL_{n_0}$ such that $\theta_0$ coincides with $\theta$ on $\GL_{n_0}(\f_{E_0})$ under this identification.  Observe that $\theta$, and hence $\theta_0$, are nontrivial on the center of $\GL_{n_0}(\f_{E_0})$.  It follows that $\theta_0$ is an outer involution of $\GL_{n_0}$ and thus that $\mathsf{G}_y^{0,\theta}\cong\On_{n_0}$.

Recall that $\rho' = \rho\otimes \phi$ and $\eta'_\theta = \eta\otimes\phi$.  
Note that $\eta_\theta (g) = \phi (g) \eta'_\theta (g)$ defines a character of exponent two of $\sf{G}^0_y (\f)$.
We therefore have
$$\langle [\theta],[\Psi]\rangle_{K^0} = 
 \dim {\rm Hom}_{K^{0,\theta}} (\rho , \eta_\theta)
 = \dim {\rm Hom}_{\mathsf{G}_y^0(\f)^\theta} ( R^\lambda_{\mathsf{T}(\f)}, \eta_\theta).$$
Thus Theorem \ref{Lusztigwithchi}
now implies $$\langle [\theta],[\Psi]\rangle_{K^0}
=\sigma (\mathsf{T}) \sum_{\gamma \in \mathsf{T}(\f)\bs \Xi_{\mathsf{T},\lambda,\eta_\theta}/\mathsf{G}_y^0(\f)^\theta} \sigma \left(Z_{\mathsf{G}_y^0}\left( (\gamma^{-1}\mathsf{T} \gamma \cap \mathsf{G}_y^{0,\theta})^\circ\right)\right).$$

Since $\langle [\theta],[\Psi]\rangle_{K^0}$ is nonzero,
by the definition of $\Xi_{\mathsf{T},\lambda,\eta_\theta}$, we see that there exists $\gamma\in  \mathsf{G}_y^0(\f)$ such that 
$(\gamma\cdot \theta) (\mathsf{T}) = \mathsf{T}$ and the summand above associated to $\gamma$ is nonzero (as are all the summands).

Suppose that $g\in G^0_{y,0}$ projects to $\gamma$.  Then $[g\cdot\theta]= [\theta]$.
Therefore, there is no essential loss in generality in replacing
$g\cdot \theta$ by $\theta$.  In other words, we may assume $g=1$ and
therefore $\theta (\mathsf{T}) = \mathsf{T}$.

\begin{lemma}
\label{lem:lift-tori}
Assuming $\theta(\mathsf{T}) = \mathsf{T}$, there exists a $\theta$-stable elliptic maximal $F$-torus $\bT'$ of
$\bG^0$ such that
\begin{enumerate}
\item $y\in A(\bG^0,\bT',F)$.
\item The image of $T'\cap G^0_{y,0}$ in $\mathsf G^0_y(\f)$ is $\mathsf T(\f)$.
\item $\bT$ and $\bT'$ are conjugate in $G^0_{y,0^+}$.
\end{enumerate}
\end{lemma}
\begin{proof}
Let $\bH = \GL_{n_0}$.  The group $\bG^0$ is isomorphic to the group
$R_{E_0/F}\bH$ obtained from $\bH$ via restriction of scalars from
$E_0$ to $F$.  As discussed in~\S\ref{sec:background}, over an algebraic closure $\overline F$ of $F$,
$$R_{E_0/F}\bH\cong\prod_{\sigma\in\Sigma}\bH,$$
where $\Sigma$ is the set of $F$-embeddings of $E_0$ in $\overline F$.  Moreover, the identification of the
$F$-group $\bG^0$ and $\prod_{\sigma\in\Sigma} \bH$ (together with
the above action of $\gal (\overline F/F)$) determines an identification of
$\cB (\bG^0,F)$ with $\cB (\bH,E_0)$.

It is easily checked that since
$[E_0 : F]$ is odd and $\theta$ is defined over $F$, $\theta$ must
stabilize each factor in the above decomposition of $\bG^0$.  Thus,
for each $\sigma\in\Sigma$, $\theta$ determines an $E_0$-automorphism
$\theta_\sigma$ of $\bH$.  In fact, $\theta_\sigma =
{}^\sigma\theta_e$, where $e\in\Sigma$ is the identity
embedding, and $^\sigma\theta_e$ is the map $x\mapsto \sigma(\theta_e(\sigma^{-1}(x)))$.

Recall that $\bT\cong R_{E/F}\GL_1 = R_{E_0/F}(R_{E/E_0}\GL_1)$.  In fact, this isomorphism is compatible with the identification of $\bG^0$ with $R_{E_0/F}\bH$ in the sense that $\bT$ can be identified with $R_{E_0/F}\bS$, where $\bS\cong R_{E/E_0}\GL_1$ is a unramified elliptic maximal $E_0$ torus of $\bH$.  The existence of a torus $\bT'$ with the above-stated properties now follows immediately from Proposition~\ref{prop:tori-lifting}. 
\end{proof}

We have thus demonstrated the following result.
\begin{proposition}\label{strongsymmetrylemma}
Suppose $\theta$ is an involution of $G$ and $\Psi$ is a cuspidal $G$-datum such that $\langle [\theta],[\Psi]\rangle_{K^0}\ne 0$.  Then there is a $\theta$-split maximal $F$-torus $\bT$ of $\bG$ with the properties given in Definition \ref{cuspidalGdatum}.
\end{proposition}

\subsection{Triviality of $\eta'_\theta$}
\label{sec:etriv}

In this section, we establish that the character $\eta'_\theta$ in the application of the theory of \cite{HM} to $(\GL_n,{\rm O}_n)$ is trivial, when $n$ is odd.

Assume $\theta$ is an orthogonal involution of $G = \GL_n(F)$, where $n$ is odd.  Let $\Psi = (\bG ,y,\rho,\vec\phi)$ be a cuspidal $G$-datum.  Let $\Phi = (\varphi , E, \{ \varphi_i\}, \{ E_i\},\iota : E\hookrightarrow M(n,F))$ be an associated Howe datum.  Let $\bT$ be the elliptic maximal $F$-torus of $\bG$ such that $\bT(F) = \iota (E^\times)$.  We may assume that $\bT$ is $\theta$-split by Proposition~\ref{strongsymmetrylemma}.

\begin{proposition}\label{etaprimeistrivial}
The character $\eta'_\theta$ is trivial.
\end{proposition}
In the toral case, this follows immediately from the fact that $K^{0,\theta}= \{ \pm 1\}$ lies in the center of $G$.

In general, the character $\eta'_\theta$ of $K^{0,\theta}$ has an expression
$$\eta'_\theta (k) = \prod_{i=0}^{d-1} \chi^{{\mathcal M}_i}(f'_i (k))$$
in the notation of~\cite{HM}.
Here the $i$th factor is given explicitly as
$$\det ({\rm Int}(k)|W_i^+)^{(p-1)/2},$$
where
 $$W^+_i = J^{i+1,\theta}/J_+^{i+1,\theta},$$
 and $J_+^{i+1}$ is a certain subgroup of finite index in $J^{i+1}$.  (See~\S3.1 in~\cite{HM}.)
In the above determinant, we are viewing $W^+_i$ as an $\f^*$-vector space, where $\f^*$ is the field of prime order contained in $\f$.  
We will show that each of the factors in the definition of $\eta'_\theta$ is trivial.  

It is more convenient to work on the Lie algebra $\g$.  The groups $J^{i+1}$ and $J^{i+1}_+$ have obvious analogues $\fJ^{i+1}$ and $\fJ^{i+1}_+$ in the Lie algebra $\g$, and it is easily seen that
$$\det ({\rm Int}(k)|W_i^+)^{(p-1)/2} = \det ({\rm Ad}(k)|\fW_i^+)^{(p-1)/2},$$
where $$\fW^+_i = \fJ^{i+1,\theta}/\fJ_+^{i+1,\theta}.$$
As above, we view $\fW^+_i$ as an $\f^*$-vector space.  In fact, the $\f^*$-linear structure on $\fW_i^+$ extends naturally to an $\f$-linear structure.
Moreover, $\Ad (k)$ is $\f$-linear.  According to a classical ``transitivity of norms'' formula (see \S7.4 in \cite{J}), we have
$$\det\nolimits_{\f^*} (\Ad (k)|\fW_i^+ ) = N_{\f/\f^*} \left( \det\nolimits_{\f}(\Ad (k)|\fW_i^+)\right).$$
To establish that  $\eta'_\theta$ is trivial, we will show that for all $i$ and for all $k\in K^{0,\theta}$, the determinant
$\det\nolimits_{\f}(\Ad (k)|\fW_i^+)$ is trivial.

\subsubsection{Some notations}

There is no loss of generality in assuming that $\bG =\bG^{i+1}$ and doing so will allow us to simplify our notations.  In particular, we take $\bG' = \bG^i$ and routinely drop subscripts and superscripts involving $i$ by using notations such as
 $\Phi = \Phi (\bG,\bT)\cup\{0\}$ and $\Phi' = \Phi (\bG',\bT)\cup\{ 0\}$.  

Note that the fact that $\bT$ is $\theta$-split implies that $\theta a = -a$, for all $a\in \Phi$.  Let $(\Phi - \Phi')^+$ be any set of representatives for the various pairs $\{ a,-a\}$ as $a$ ranges over $\Phi -\Phi'$.  For each $a\in \Phi - \Phi'$, we have the 1-dimensional space
$$\bfr g_a^\theta = (\bfr g_a + \bfr g_{-a})^\theta.$$
For any extension $K$ of $F$ contained in $\overline F$, 
let
\begin{eqnarray*}
\fW (K)&=& \bigoplus_{a\in \Phi - \Phi'} \bfr g_a(K)_{y,s:s^+},\\
\fW^+ (K)&=& \bigoplus_{a\in ( \Phi -\Phi')^+} \bfr g^\theta_a(K)_{y,s:s^+}.
\end{eqnarray*}
Let $\dot E/F$ denote the Galois closure of $E/F$ in $\overline F$.  
Then $\fW$ and $\fW^+$ are the spaces of ${\rm Gal}(\dot E/F)$-fixed points in $\fW(\dot E)$ and $\fW^+(\dot E)$, respectively.

\subsubsection{The structure of the proof}
Let $\delta : K^{0,\theta} \to \f^\times$ be the map
$$k\mapsto \det\nolimits_{\f} (\Ad (k) | \fW^+).$$
To show $\delta$ is trivial, first observe that it is trivial on $K^{0,\theta}_+$.  Abbreviate $\f_{E_0}$ by $\f_0$.  In~\S\ref{sec:reduction}, we observed that 
$$K^{0,\theta}/K^{0,\theta}_+\cong \On_{n_0} (\f_0).$$
Note that $\delta$ must be trivial on the negative of the identity matrix.  Thus it suffices to show that $\delta$ is trivial as a homomorphism
$\SOn_{n_0}(\f_0) \to \f^\times$.

Our basic strategy can now be described as follows.  Let $F'$ be a (unique up to isomorphism) unramified quadratic extension of $F$.  Let  $\f' = \f_{F'}$ and let $\f'_0 = \f_{F'E_0}$.  Taking $F'$-rational points, we show that $\delta$ has a natural extension to a homomorphism
$$\delta' : \SOn_{n_0}( \f'_0) \to (\f')^\times .$$
Then the triviality of $\delta$ follows from the fact (shown below) that $\SOn_{n_0}( \f_0)$ is contained in the commutator subgroup of $\SOn_{n_0}(\f_0')$.

\subsubsection{The spinor norm}
Let $p$ be an odd prime and let $\F_p$ denote the field of order $p$.
As in~\S \ref{sec:genlusztig}, for any power $q$ of $p$, let $\F_q$ denote the finite field of order $q$ (inside a fixed algebraic closure of $\F_p$).  Let $\nu : \On(n_0, \F_q)\to \F_q^\times /(\F_q^\times)^2$ be the spinor norm.  
Recall that an element of $\On(n_0, \F_q)$ lies in the kernel of $\nu$ precisely if it can be expressed as a product of reflections $r_{v_1}\cdots r_{v_m}$ through anisotropic vectors $v_1,\ldots ,v_m\in\F_q^{n_0}$ such that 
$$Q(v_1)\cdots Q(v_m)\in(\F_q^\times)^2,$$
where $Q$ is the quadratic form on $\F_q^{n_0}$ that is used to define $\SOn_{n_0}$.
It is well known that the commutator subgroup of $\SOn_{n_0}(\F_q)$ is the group $B_k (q)$
consisting of the elements in the kernel of $\nu$ that also lie in $\SOn_{n_0}(\F_q)$, where
 $k = (n_0 -1)/2$.
The group $B_k (q)$ is also the commutator subgroup of $\On_{n_0}(\F_q)$ and it has index two in $\SOn_{n_0}(\F_q)$.
 
 \begin{lemma}
 \label{lem:spinor}
For any power $q$ of $p$, $\SOn_{n_0}(\F_q)$ is contained in the commutator subgroup $B_k (q^2)$ of $\SOn_{n_0}(\F_{q^2})$.
 \end{lemma}

\begin{proof}
Given $g\in \SOn_{n_0}(F_q)$, we can write $g = r_{v_1}\cdots r_{v_{m}}$, where each $v_i$ in $\F_q^{n_0}$ is anisotropic with respect to the quadratic form defining $\SOn_{n_0}$, and $r_{v_i}$ is the associated reflection.

Let $\widetilde Q$ be the obvious extension of the above quadratic form to $\F_{q^2}^{n_0}$.  Then $\widetilde Q(v_i)\in \F_q^\times \subset ( \F_{q^2}^\times )^2$.  Therefore, if $\nu'$ is the spinor norm on $\On_{n_0}(\F_{q^2})$ then $\nu' (g) =1$.
\end{proof}

A general reference for the material in this section is \cite{Lm}.

\subsubsection{Extension of scalars}

Let $\varepsilon$ be a unit in $F$ whose image in the residue field $\f$ generates $\f^\times$.  Let $F' = F [\sqrt{\varepsilon}]$ and $E' = \dot E [\sqrt{\varepsilon}]$.  Then $E'/\dot E$ and $F'/F$ are unramified quadratic extensions and $E' = \dot EF'$.
Note that restriction from $E'$ to $\dot E$ defines an isomorphism 
$${\rm Gal}(E'/F')\cong {\rm Gal}(\dot E/F)$$
whose inverse is $$\alpha\mapsto (x+y\sqrt{\varepsilon}\mapsto \alpha (x) +\alpha (y)\sqrt{\varepsilon}).$$
Note that $\fW (F')$ and $\fW^+ (F')$ are the spaces of ${\rm Gal}(E'/F')$-fixed points in 
$\fW (E')$ and $\fW^+ (E')$, respectively.
All of these spaces may be regarded as $\f'$-vector spaces and we have
\begin{eqnarray*}
\fW (F')&=& \fW \otimes_{\f} \f',\\
\fW^+ (F')&=&\fW^+ \otimes_{\f} \f'.
\end{eqnarray*}

Let  $$K^0 (F') = \bG^0(F')_{[y]}, \quad K^0(F')_+ = \bG^0 (F')_{y,0^+}.$$  
Then $$K^0(F')^\theta = \bG^0(F')_{y,0}^\theta, \quad K^0(F')_+^\theta = \bG^0 (F')^\theta_{y,0^+}.$$
By the discussion in~\S\ref{sec:reduction}, we have $$K^0(F')^\theta /K^0 (F')^\theta_+ = \bG^0(F')^\theta_{y,0:0^+} \cong \On_{n_0}(\f_0').$$

For $k\in K^0 (F')^\theta$, define a homomorphism $\delta' : K^0 (F')^\theta \to (\f')^\times$ by
$$\delta' (k) = \det\nolimits_{\f'} (\Ad (k) | \fW^+(F')).$$
We regard $\delta'$ also as a homomorphism
$$\delta' : \SOn_{n_0}(\f'_0) \to (\f')^\times.$$
Observe that since $\fW^+ (F')=\fW^+ \otimes_{\f} \f'$, the restriction of $\delta'$ to $\SOn_{n_0}(\f_0)$ is $\delta$.
Since $\SOn_{n_0}(\f_0)$ is contained in the commutator subgroup of $\SOn_{n_0}(\f'_0)$ by Lemma~\ref{lem:spinor}, $\delta$ is trivial.  It follows that $\eta'_\theta$ must be trivial.

\subsection{Lusztig's theory for our examples}

To simplify notations, we assume in this section that $n= n_0$.  Later, we use the results of this section with $n$ replaced by $n_0$.  Our objective is to apply the results of \S \ref{sec:genlusztig} to the finite groups that arise from the tame supercuspidal representations consdered in this paper.  What we do turns out to be a routine generalization of \S2 in \cite{HzM} analogous to our generalization of the theory in \cite{L}.

We resume the notations of \S \ref{sec:genlusztig} with $\F_q = \f_0$ and $\bG = \GL_{n}$ (where $n$ is an odd integer greater than 1).
Let $\theta$ be the involution $\theta (g) = {}^t g^{-1}$ of $G$.
Then $\bG^\theta = {\rm O}_{n}$ and $(\bG^\theta)^\circ = {\rm SO}_{n}$.  Let $\cJ$ be the set of all $\theta$-split  maximal tori in $\bG$.  The group $(\bG^\theta)^\circ$ acts transitively on $\cJ$ by conjugation.  (See \S1.5 \cite{L}.)

Let $\bT$ be a $\theta$-stable elliptic maximal torus in $\bG$.  Let $\lambda$ and $\chi$ be complex characters of $T$ and $G^\theta$, respectively.  Assume that $\lambda$ is nonsingular (in the sense of \cite{L}). Let $\Xi_\bT$ denote the set
of all $g\in G$ such that $(g\cdot \theta)(\bT) = \bT$.  Like $\Xi_{\bT,\lambda,\chi}$, this set is a union of double cosets in $T\bs G/G^\theta$.

\begin{lemma}\label{HMaolemma}  The set $\Xi_\bT$
consists of a single double coset in $T\bs G/G^\theta$.
The set $\Xi_{\bT,\lambda,\chi}$ is empty unless $\lambda (-1) = \chi (-1)$ in which case it equals $\Xi_\bT$.
\end{lemma}

\begin{proof}
The first assertion is Lemma 2 of \cite{HzM}.  As stated, this lemma only applies to a certain specific elliptic maximal $\F_q$-torus of $\bG$.  However, the lemma holds for any such torus since all such tori are conjugate in $G$.  We now prove the second assertion (which generalizes Lemma~1 of \cite{HzM}).

Suppose $g\in \Xi_\bT$.  Then $g^{-1} \bT g$ is $\theta$-stable and hence $\theta$-split by Lemma \ref{thetastabisthetasplit} (which applies equally well when the local field $F$ is replaced by the finite field $\F_q$).  Hence $g^{-1}\bT g\in \cJ$.  Let $\bA$ be the $\theta$-stable (hence $\theta$-split) maximal $\F$-torus of $\bG$ consisting of the diagonal matrices.  We may choose $h\in (\bG^\theta)^\circ$ such that $g^{-1}\bT g = h\bA h^{-1}$.  It follows that $G^\theta\cap g^{-1} T g\subset h (\bG^\theta \cap \bA) h^{-1}$.  But the elements of $\bG^\theta\cap \bA$ are diagonal matrices whose diagonal entries are $\pm 1$.  Thus the squares of all elements of $G^\theta \cap g^{-1} Tg$ are trivial.  Since $g^{-1}Tg\cong \F_{q^n}^\times$, we deduce that $G^\theta \cap g^{-1}Tg =\{ \pm 1\}$.  Since $\varepsilon_{g^{-1}\bT g}(\pm 1) =1$, we see that $g\in \Xi_{\bT,\lambda, \chi}$ if and only if $\lambda (-1) = \chi (-1)$.  But the latter condition does not depend on $g$.  Therefore, if it is satisfied we have $\Xi_{\bT,\lambda,\chi}=\Xi_\bT$ and if it is not satisfied $\Xi_{\bT,\lambda,\chi}$ is empty.
\end{proof}

\begin{proposition}\label{orthLusztigwithchi}
Suppose  $\theta$ is an orthogonal involution of $G$ and $\bT$ is an elliptic maximal $\F_q$-torus in $\bG$.  If $\lambda$ is a  character of $T$ and  $\chi$ is a character of $G^\theta$ then
$$\frac{1}{|G^\theta|} \sum_{h\in G^\theta}  R_{\bT,\lambda} (h)\ \chi(h)
=\begin{cases}1,&\text{if }\lambda(-1) = \chi(-1),\\ 0,&\text{otherwise.}
\end{cases}$$
\end{proposition}

\begin{proof}
We first assume that $\theta$ is chosen as above.
Theorem \ref{Lusztigwithchi} says
$$\frac{1}{|G^\theta|} \sum_{h\in G^\theta}  R_{\bT,\lambda} (h)\ \chi(h)
=\sigma (\bT) \sum_{g\in T\bs \Xi_{\bT,\lambda,\chi}/G^\theta} \sigma \left(Z_\bG\left( (g^{-1}\bT g \cap \bG^\theta)^\circ\right)\right).$$
If $\lambda(-1)\ne \chi (-1)$ then our claim follows from Lemma \ref{HMaolemma} since the latter sum over $T\bs \Xi_{\bT,\lambda,\chi}/G^\theta$ is an empty sum.  Now assume $\lambda(-1) = \chi (-1)$.  Then Lemma \ref{HMaolemma} implies that $\Xi_{\bT,\lambda,\chi} = Tg_0 G^\theta$.  We have $\sigma (\bT) =-1$ and 
$$ \sigma \left(Z_\bG\left( (g^{-1}\bT g \cap \bG^\theta)^\circ\right)\right) =
\sigma (\bG) =-1.$$  This establishes our claim for the given $\theta$.  The case of general orthogonal involutions follows upon applying an inner automorphism to the formula in the special case already proven.
\end{proof}

\subsection{Main results}

We now prove the main theorem:

\begin{theorem}\label{maintheorem}
Let $\pi$ be an irreducible tame supercuspidal representation of $G$ with central character $\omega$ and let $\theta$ be an orthogonal involution of $G$.  Then $\pi$  is $G^\theta$-distinguished precisely when $\theta$ lies in  $\Theta_J$ and
 $\omega (-1)=1$.  When $\pi$ is $G^\theta$-distinguished, the dimension of $\Hom_{G^\theta} (\pi ,1)$ is one. 
If $\pi$ is associated to  an $F$-admissible quasicharacter $\varphi$
then the condition $\omega (-1) =1$ can also be stated as $\varphi (-1)=1$.  Similarly, if $\pi$ is associated to a cuspidal $G$-datum $\Psi = (\vec\bG ,y,\rho,\vec\phi)$ and if $\omega'$ is the central character of $\rho'= \rho\otimes\phi$ then $\omega (-1) =1$ can also be stated as $\omega' (-1) =1$.
\end{theorem}

\begin{proof}
Let $\Theta$ be a $G$-orbit of orthogonal involutions of $G$ and let $\Psi$ be a cuspidal $G$-datum to which $\pi$ is associated.  From \S \ref{sec:correction}, we have
$$\langle \Theta, \Psi\rangle_G = \sum_{[\theta]\sim[\Psi]} m_{K^0} ([\theta])\ \langle [\theta] , [\Psi]\rangle_{K^0},$$
which simplifies to
$$\langle \Theta, \Psi\rangle_G = \sum_{[\theta]\sim[\Psi]} \langle [\theta] , [\Psi]\rangle_{K^0},$$  
according to Lemma \ref{mKthetaistrivial}.

Suppose we have a nonzero summand $\langle [\theta] , [\Psi]\rangle_{K^0}$.
According to Proposition \ref{strongsymmetrylemma}, there exists a $\theta$-split maximal $F$-torus $\bT$ with the properties described in Definition~\ref{cuspidalGdatum}.  Proposition \ref{refinednoddresult}
implies that the $G$-orbit $\Theta_J$  is the unique $G$-orbit of orthogonal involutions of $G$ that contains an involution $\theta'$ for which $\bT$ is $\theta'$-stable (and hence $\theta'$-split).   Therefore, the existence of a nonzero summand implies that $\Theta = \Theta_J$.  In other words, since $\langle \Theta ,\Psi\rangle_G$ is nonzero then $\Theta = \Theta_J$.

Now fix a maximal torus $\bT$ as above.  
Proposition \ref{refinednoddresult} states that the set of all $\theta'\in \Theta_J$ such that $T$ is $\theta'$-split
comprises a single $T$-orbit in $\Theta_J$.  Since $T\subseteq K$, we see that there can be at most one nonzero summand.  Let $\theta$ be an element of the unique orbit parametrizing this summand.
Then
$$\langle \Theta , \Psi\rangle_G = \langle [\theta] , [\Psi]\rangle_{K^0}.$$
By Proposition \ref{etaprimeistrivial}, $\eta'_\theta$ is trivial and thus
$$\langle [\theta ] , [\Psi]\rangle_{K^0} =\dim \Hom_{K^{0,\theta}}(\rho', \eta'_\theta) = \dim \Hom_{K^{0,\theta}}(\rho', 1).$$ 

Assume  that the datum $\Psi$ is toral.
Then $\rho =1$ and $\rho' = \rho\otimes \phi = \phi$.
If $\Psi$ comes from a Howe datum $\Phi$ and we identify $T$ with $E^\times$ then the quasicharacter $\phi$ of $T$ corresponds to the $F$-admissible quasicharacter $\varphi$ of $E^\times$.

But $K^{0,\theta} = \{ \pm 1\}$, $\rho'=\phi$
so
$$\Hom_{K^{0,\theta}} (\rho',1) = \Hom_{ \{ \pm 1\} }(\phi , 1),$$ 
and, in the toral case, we have
 $$\langle \Theta ,\Psi\rangle_G =
\begin{cases}1,&\text{if }\phi (-1)=1,\\
0,&\text{if }\phi(-1) = -1,
\end{cases}
$$
or, in other words,
$$\langle \Theta ,\Psi\rangle_G =
\begin{cases}1,&\text{if }\omega' (-1)=1,\\
0,&\text{if }\omega' (-1) = -1,
\end{cases}
$$ where $\omega'$ is the central character of $\rho'$.

Now assume $\Psi$ is not toral.  Referring back to the discussion in \S \ref{sec:reduction}, we obtain the formula
$$\langle \Theta,\Psi\rangle_G= \langle [\theta ],[\Psi]\rangle_{K^0} = \dim {\rm Hom}_{\mathsf{G}_y^0(\f)^\theta} ( R^\lambda_{\mathsf{T}(\f)}, \eta_\theta).$$
We now apply Proposition \ref{orthLusztigwithchi} to the $\f_0$-group $\mathsf G_y^0$ using the fact that $\eta_\theta = \phi |G^{0,\theta}_{y,0}$
and  again we obtain
$$\langle \Theta ,\Psi\rangle_G =
\begin{cases}1,&\text{if }\omega' (-1)=1,\\
0,&\text{if }\omega' (-1) = -1,
\end{cases}
$$ where $\omega'$ is the central character of $\rho'$.

The assertions of the theorem now follow directly for representations $\pi$ that are $G^\theta$-distinguished.  It remains to show that our necessary conditions for $G^\theta$-distinction are also sufficient conditions.
Now suppose $\theta\in \Theta_J$ and $\Psi = (\vec\bG ,y,\rho,\vec\phi)$ is a cuspidal $G$-datum with $\omega' (-1) =1$.  
Let $E/F$ be a tame extension of degree $n$ appearing in a Howe datum associated to $\Psi$ (as in~\S\ref{sec:attaching}).
The maximal torus $\bT$ appearing in Definition~\ref{cuspidalGdatum} must be isomorphic to $R_{E/F}\GL_1$.

Since $\theta\in \Theta_J$, Proposition~\ref{refinednoddresult} implies that there exists a $\theta$-stable embedding of $E$ in $M(n,F)$.  The results of~\S\ref{sec:embeddings} show that after conjugating by an appropriate element of $G$, we may assume that $\bT$, and hence $[y]$, are $\theta$-stable.  Moreover, $\bT$ must be $\theta$-split by ~\ref{refinednoddresult}.  The same must be true for $\bZ^0\subset \bT$.  By Corollary~\ref{cor:theta-symmetric}, we must have
\begin{enumerate}
\item [(1)] $\theta$ stabilizes $\vec\bG$,
\item [(2)] $\phi\circ\theta=\phi^{-1}$.
\end{enumerate}
In particular, (1) implies that $\theta$ must stabilize $\bG^0$, hence $K^0$, by Lemma~\ref{lem:K^0}.  Moreover, (2) implies that for $g\in K^{0,\theta}_+$, we have $\phi(g) = \phi(\theta(g)) = \phi(g)^{-1}$.  Thus $\phi(g)=\pm 1$, and since $K^{0,\theta}_+$ is a pro-$p$ group, we must have $\phi|K^{0,\theta}_+ = 1$.   We have therefore shown that $[\theta]\sim[\Psi]$.  As in the above discussion, we see that the assumption that $\omega' (-1) =1$ then implies that
$\langle [\theta] ,[\Psi] \rangle_{K^0} = 1$ and thus $\pi$ is $G^\theta$-distinguished.
\end{proof}

\appendix
\section{Lifting Tori from $\mathsf G_y$ to $\bG$}
\label{sec:lifting-tori}
Let $\gF$ denote the algebraic closure of $\f$ and let $\Fun$ denote the maximal unramified extension of $F$.
The following result
is a direct consequence of~\cite[\S2]{D}.
\begin{lemma}
\label{lem:lifting-tori}
  Let $\bG$ be a reductive group over $F$.  Suppose that $\bT$ is an elliptic maximal unramified $F$-torus of $\bG$ and $y$ is a vertex in $A(\bG, \bT, F)$.
\begin{enumerate}
  \item  The torus $\bT$ determines a minisotropic maximal $\f$-torus $\mathsf T$ of $\mathsf
  G_y$ such that the image of $\bT(\Fun)\cap\bG(\Fun)_{y,0}$ in $\mathsf G_y(\gF)$ is $\mathsf
  T(\gF)$.
  \item If $\bT'$ is another elliptic maximal unramified $F$-torus of $\bG$ such that $y\in A(\bG, \bT', F)$, and $\mathsf T'$ is the associated torus of $\mathsf G_y$, then $\mathsf T' = \mathsf T$ if
  and only if $\bT'$ is conjugate to $\bT$ by an element of $G_{y,0^+}$.
\item Every minisotropic maximal $\f$-torus of $\mathsf G_y$ is associated to some elliptic maximal unramified $F$-torus of $\bG$ in this way.
\end{enumerate}
\end{lemma}

Suppose now that $\theta$ is an involution of $G$.  The following result is an analogue of Lemma~\ref{lem:lifting-tori} (3) for $\theta$-stable tori.

\begin{lemma}
\label{lem:stable-tori}
Let $\bG$ be a reductive algebraic group defined over $F$, and let
$\theta$ be an $F$-involution of $\bG$.  Let $y$ be a vertex in $\cB (\bG,F)$ and suppose
that $\theta([y]) = [y]$.  Let $\mathsf T$ be a $\theta$-stable
minisotropic maximal $\f$-torus of $\mathsf G_y$.
Then there exists a $\theta$-stable elliptic maximal unramified $F$-torus $\bT$ of $\bG$
such that $y\in A(\bG, \bT, F)$ and the image of $\bT(\Fun)\cap\bG(\Fun)_{y,0}$ in $\mathsf G_y(\gF)$ is $\mathsf T(\gF)$.
\end{lemma}

\begin{proof}
  Let $\mathcal S_0$ be the set of elliptic maximal $F$-unramified tori
  $\bS$ of $\bG$ such that 
\begin{itemize}
\item $A(\bG, \bS, F)$ contains $y$,
\item  The maximal torus of $\mathsf G_y$ determined by $\bS$ is
  $\mathsf T$.
\end{itemize}
This set is nonempty by Lemma~\ref{lem:lifting-tori} (3).  
Note that $\mathcal S_0$ is
$\theta$-stable.  By Lemma~\ref{lem:lifting-tori} (2), $G_{y,0^+}$ acts transitively by conjugation on
$\mathcal S_0$.  We may therefore topologize $\mathcal S_0$ by giving it the quotient topology inherited from $G_{y,0^+}$.  With this topology, it is clear that $G_{y,0^+}$ acts continuously on $\mathcal S_0$.  Moreover, $\mathcal S_0$ is compact and metrizable.

  Let $X = \{ r\in\R : r\geq 0, G_{y,r}\neq G_{y,r^+}\}$.  The elements
  of $X$ can be written as a sequence $r_0, r_1,r_2,\ldots $, where
  $r_0 = 0$.  We now inductively define a nested sequence of compact subsets of $\mathcal S_0$.  Suppose we have already defined a sequence $\mathcal S_0,\mathcal S_1,\ldots,\mathcal S_i$ of compact subsets of $\mathcal S_0$ such that each $\mathcal S_j$ is a $\theta$-stable
  orbit of $G_{y,r_j^+}$ in $\mathcal S_0$.  (To begin with, note that this is true for $\mathcal S_0$.)  We claim that there is
  some $\theta$-stable orbit $\mathcal S_{i+1}$ of $G_{y,r_{i+1}^+}$ in
  $\mathcal S_i$.

  Consider the collection $\mathcal T$ of $G_{y,r_{i+1}^+}$-orbits in
  $\mathcal S_i$.  The group $G_{y,r_i^+}$ acts transitively on
  $\mathcal T$, and therefore $\mathcal T$ has size
  dividing $[G_{y,r_i^+}:G_{y,r_{i+1}^+}]$, which is a power of $q$, hence odd.
  Since $\theta$ acts as a permutation of $\mathcal T$ of order
  dividing $2$, some $G_{y,r_{i+1}^+}$-orbit $\mathcal
  S_{i+1}\in\mathcal T$ must be fixed by $\theta$, proving the claim.

  Note that the $\mathcal S_i$ form a nested sequence of nonempty compact
 subspaces.  Hence, $\mathcal S = \bigcap_i \mathcal S_i$ is
  nonempty.  Moreover, $\mathcal S$ is $\theta$-stable and is
  contained in a single $G_{y,r_i^+}$-orbit for each $i$, hence must
  be a singleton.  In other words $\mathcal S$ consists of a
  single $\theta$-stable torus.
\end{proof}

Now consider the following situation.
Let $E$ be a finite extension of $F$.  Let $\bH$ be an unramified reductive group defined over $E$, and let $\bG$ be the group $R_{E/F}\bH$ obtained from $\bH$ via restriction of scalars.  Let $\theta_0$ be an $E$-involution of $\bH$.  Then $\theta_0$ naturally determines an $F$-involution of $\bG$.  Let $\bS$ be an elliptic unramified maximal $E$-torus in $\bH$ and let $\bT$ be the torus $R_{E/F}\bS$ in $\bG$.  Then $\bT$ is an elliptic maximal torus of $\bG$ which contains a maximal unramified torus of $\bG$.  Let $y$ be a vertex in $A(\bG,\bT,F)$ such that $\theta([y]) = [y]$.  Since  $A(\bG,\bT,F) = A(\bH,\bS,E)$, we can also view $y$ as a $\theta_0$-fixed point of $A(\bH,\bS,E)$.  Note that $\theta$ descends to an $\f$-involution of $\mathsf G_y$ (which we will also denote by $\theta$).  Similarly, $\theta_0$ descends to an $\f_E$-involution of the $\f_E$-group $\mathsf H_y^E$.

\begin{proposition}
\label{prop:tori-lifting}
In the above situation, if the $\f$-torus $\mathsf T$ in $\mathsf G_y$ determined by $\bT$ is $\theta$-stable, then there is an element $g\in G_{y,0^+}$ such that $g\bT g^{-1}$ is $\theta$-stable.
\end{proposition}
\begin{proof}
Let $K/F$ be the maximal unramified subextension of $E/F$.  Let $\widetilde\bH = R_{E/K}\bH$ and $\widetilde\bS = R_{E/K}\bS$.  By the transitivity of restriction of scalars, $\bG = R_{K/F}\widetilde\bH$ and $\bT = R_{K/F}\widetilde\bS$.  Note that $\theta_0$ determines a $K$-involution $\tilde\theta_0$ of $\widetilde\bH$, which descends to an $\f_E$-involution of the $\f_E$-group $\widetilde{\mathsf H}^K_y$.

Let $\widetilde{\mathsf S}$ be the $k_E$-torus in $\widetilde{\mathsf H}^K_y$ determined by $\widetilde\bS$.  Since $K/F$ is unramified, it follows that $\mathsf G_y =  R_{\f_E/\f}\widetilde{\mathsf H}^K_y$ and $\mathsf T =  R_{\f_E/\f}\widetilde{\mathsf S}$.  Moreover, the involution of $\mathsf G_y$ determined by the involution $\tilde\theta_0$ of $\widetilde{\mathsf H}^K_y$ is precisely $\theta$.  Since $\mathsf T$ is $\theta$-stable, it follows that $\widetilde{\mathsf S}$ must be $\tilde\theta_0$-stable.

 Since $E/K$ is totally ramified, it follows from Lemma~2.1.1 of~\cite{AD} that $\widetilde\bH(K)_y = \bH(E)_y$ and hence that $\widetilde{\mathsf H}^K_y = \mathsf H^E_y.$  Similarly, $\widetilde{\mathsf S} = \mathsf S$.  It is easily seen that the actions of $\tilde\theta_0$ on $\widetilde{\mathsf H}^K_y$ and $\theta_0$ on $\mathsf H^E_y$ coincide under the above identification.  Thus since $\widetilde{\mathsf S}$ is $\tilde\theta_0$-stable, it follows that $\mathsf S$ is $\theta_0$-stable.
 
By Lemmas~\ref{lem:lifting-tori} and~\ref{lem:stable-tori}, there exists $g\in\bH(E)_{y,0^+}$ such that $\bS' = g\bS g^{-1}$ is $\theta_0$-stable.  Let $\bT' = R_{E/F}\bS'\subset\bG$.  Then $\bT' = g\bT g^{-1}$, where $g$ here is viewed as an element of $G_{y,0^+} = \bH(E)_{y,0^+}$.  Moreover, since $\bS'$ is $\theta_0$-stable, $\bT'$ must be $\theta$-stable.
\end{proof}

\bibliographystyle{amsalpha}

\end{document}